\documentclass[letterpapper,twoside]{article}
\usepackage{graphicx} % Required for inserting images
\usepackage[margin=2.5cm]{geometry}
\usepackage[export]{adjustbox}

\usepackage{amsmath,amsfonts,amssymb,amsthm,mathabx,mathtools}
\usepackage[utf8]{inputenc}
\usepackage{graphicx,caption,subcaption}
\graphicspath{{images/}}
\usepackage{xcolor}

\usepackage{titlecaps}
\Addlcwords{a is with of for on in are and its to the by under around an}

\newcommand{\mytitle}{ \titlecap{%
Effective Stability of Near-Rectilinear Halo Orbits\\ in the Earth-Moon System
}}

\newcommand{\myshorttitle}{%\titlecap{%
Effective Stability of NRHOs in the E-M system
}
\usepackage[hidelinks]{hyperref}
\hypersetup{
bookmarksnumbered=true,
bookmarksopen=true,
pdftitle={\myshorttitle{}},
pdfauthor={J. Gimeno and L. Peterson}
}
\usepackage{orcidlink}

\usepackage{xparse}
\newcommand{\titlecapheading}[1]{%
  \texorpdfstring{\titlecap{#1}}{#1}%
}
\let\oldsection\section
\RenewDocumentCommand{\section}{s o m}{%
  \IfBooleanTF{#1}
    {\oldsection*{\titlecapheading{#3}}}
    {%
      \IfNoValueTF{#2}
        {\oldsection{\titlecapheading{#3}}}
        {\oldsection[\titlecapheading{#2}]{\titlecapheading{#3}}}%
    }%
}
\let\oldsubsection\subsection
\RenewDocumentCommand{\subsection}{s o m}{%
  \IfBooleanTF{#1}
    {\oldsubsection*{\titlecapheading{#3}}}
    {%
      \IfNoValueTF{#2}
        {\oldsubsection{\titlecapheading{#3}}}
        {\oldsubsection[\titlecapheading{#2}]{\titlecapheading{#3}}}%
    }%
}
\let\oldsubsubsection\subsubsection
\RenewDocumentCommand{\subsubsection}{s o m}{%
  \IfBooleanTF{#1}
    {\oldsubsubsection*{\titlecapheading{#3}}}
    {%
      \IfNoValueTF{#2}
        {\oldsubsubsection{\titlecapheading{#3}}}
        {\oldsubsubsection[\titlecapheading{#2}]{\titlecapheading{#3}}}%
    }%
}

\usepackage{comment} 

\newcommand{\imag}{\texttt{i}}
\newcommand{\muem}{\mu_{\texttt{EM}}}
\newcommand{\dop}{\mathrm{D}} % derivative operator
\newtheorem{theorem}{Theorem}%  meant for continuous numbers
\newtheorem{corollary}[theorem]{Corollary}
\newtheorem{proposition}[theorem]{Proposition}% 
\newtheorem{lemma}[theorem]{Lemma}% 

\theoremstyle{remark}

\title{\mytitle{}}
\author{Joan Gimeno\,\orcidlink{0000-0002-8707-6379}\textsuperscript{(1)} \and 
Luke T. Peterson\,\orcidlink{0000-0003-2648-8094}\textsuperscript{(2,\textasteriskcentered)}
}
\date{\today}

\begin{document}

\maketitle

{\small
\begin{itemize}
\renewcommand{\itemsep}{0pt}
\item [(1)] Departament de Matem\`atiques i Inform\`atica, %
Universitat de Barcelona, %
Gran Via de les Corts Catalanes, 585, 08007 Barcelona, %
Spain, \verb+joan@maia.ub.es+ 
\item [(2)] Department of Aerospace Engineering and Engineering Mechanics, %
University of Texas at Austin, %
2617 Wichita Street North, Austin, TX 78712, USA %
\verb+ltp@utexas.edu+ 
\item[\textsuperscript\textasteriskcentered] {\footnotesize
  Corresponding Author}
\end{itemize}
}

\begin{abstract}
Near-rectilinear halo orbits (NRHOs) around Earth-Moon $L_2$ in the Circular Restricted 3-Body Problem (CR3BP) exhibit a complex dynamical landscape, featuring a band of normally elliptic orbits embedded within regions of strong instability. This coexistence of stable and unstable dynamics, amplified by the numerical sensitivity associated with close lunar passages, makes the long-term behavior of trajectories near NRHOs a delicate and intrinsically nonlinear problem. Understanding the effective stability of these elliptic orbits is therefore a critical challenge, lying at the intersection of local normal form theory and global instability mechanisms. 

To quantify finite-time confinement, we formulate a rigorous framework for effective stability using discrete Poincaré maps. By employing jet transport to compute high-order Taylor expansions, we construct explicit polynomial normal forms. We derive discrete Nekhoroshev-type estimates by identifying the optimal normalization order, which balances the asymptotic convergence of the map's analyticity domain against the cumulative penalty of low-order small divisors. 

Applying this framework to the Earth-Moon system, we map the resulting geometric limits directly into physical spatial coordinates. Crucially, we demonstrate that for practical mission lifetimes (e.g., 10-50 years), the required stability is vastly shorter than the characteristic Nekhoroshev accumulation time. Consequently, the effective stability region is not constrained by time-dependent exponential drift, but is instead governed entirely by the maximum analytical domain of the optimized normal form. These derived spatial envelopes establish explicit geometric boundaries for the intrinsic local stability of elliptic NRHOs, providing a rigorous mathematical characterization of their nonlinear confinement within the CR3BP. 

\end{abstract} 

\vspace{0.5cm}

{\footnotesize
\noindent \textbf{Keywords:} Near-Rectilinear Halo Orbits; Circular Restricted Three-Body Problem; Effective Stability; Discrete Normal Forms; Jet Transport  
}

\pagestyle{myheadings}
\markboth{\myshorttitle{}}{ J. Gimeno and L. Peterson }
\newpage
\phantomsection\pdfbookmark[1]{\contentsname}{toc}
{\small\tableofcontents}

\newpage 

%%%%%%%%%%%%%%%%%%%
% I. Introduction %
%%%%%%%%%%%%%%%%%%%
\section{Introduction}

% I.A. Literature review
\subsection{Motivation and Background}

% I.A.1.  Near-rectilinear Halo Orbits in the Earth-Moon System 
\subsubsection{Near-Rectilinear Halo Orbits in the Earth-Moon System}
Identifying orbits with advantageous operational properties is crucial for future long-term cislunar infrastructure that will serve as communication hubs, science laboratories, and staging points for deep space missions. Premier candidates for these next-generation architectures are near-rectilinear halo orbits (NRHOs) around the Earth-Moon (EM) $L_2$ point \cite{zimovan2017near,parrish2020near}. As a family, these orbits provide excellent access to the lunar surface paired with minimal eclipsing by the Earth's shadow, making them ideal for sustained operations \cite{howell1984almost,williams2017targeting}. However, the present a highly complex dynamical landscape. When modeled within the circular restricted 3-body problem (CR3BP)--a Hamiltonian system--the majority of Earth-Moon $L_2$ halo orbits are linearly unstable. They are governed by prominent stable and unstable manifolds that dictate the chaotic flow of the local phase space. Yet, embedded within this highly unstable family exists a specific, continuous band of NRHOs that exhibit purely elliptic normal behavior (see Figure \ref{fig:IntroFigure}). These orbits within the linearly stable region offer utility in mission planning as ideal long-term staging orbits, while also generating interest from a dynamical systems perspective due to the presence of Lagrangian ``sticky'' tori that leads to slow diffusion times \cite{perry1994kam}. In particular, this relatively small band of normally elliptic halo orbits, surrounded by a largely hyperbolic family, begs the fundamental question: how can we rigorously quantify the long-term stability of these normally elliptic orbits?

\begin{figure}[h!]
    \centering
    \begin{subfigure}{.45\textwidth}
        \centering
        \includegraphics[width=\linewidth]{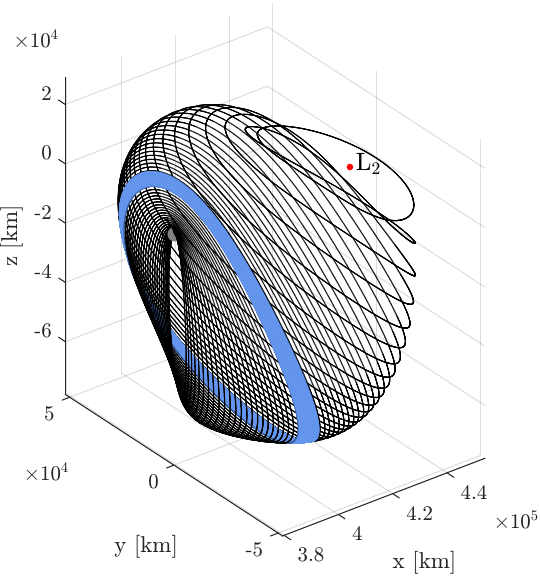}
        \caption{EM $L_2$ Southern Halo Orbits}
        \label{fig:OrbitPlot}
    \end{subfigure}%
    \hfill
    \begin{subfigure}{.53\textwidth}
        \centering
        \includegraphics[width=\linewidth]{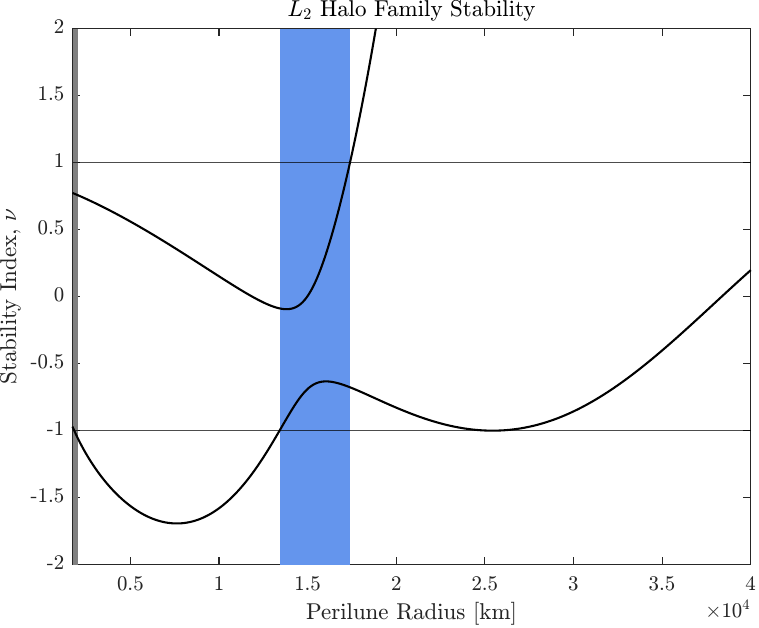}
        \caption{Stability of $L_2$ Halo Orbits}
        \label{fig:StabilityPlot}
    \end{subfigure}
    \caption{Normally elliptic $L_2$ halo orbits (blue) surrounded by linearly unstable orbits. Note that the gray bar in the right figure indicates the radius of the Moon. Right figure inspired by Spree n\'ee Zimovan et al. \cite{zimovan2017near}.}
    \label{fig:IntroFigure}
\end{figure}

% I.A.2. Effective Stability in Hamiltonian and Nearly-Integrable Systems 
\subsubsection{Effective Stability in Hamiltonian and Nearly-Integrable Systems}
The study of long-term confinement in nearly-integrable systems is historically bifurcated into two regimes: perpetual stability and effective stability \cite{delshams1996effective}. KAM (Kolmogorov-Arnold-Moser) theory guarantees perpetual stability for a large measure of initial conditions, proving the survival of invariant tori under small perturbations \cite{kolmogorov1954conservation,Arnold63a,moser1962invariant,Llave01}. However, KAM tori persist for a nowhere-dense Cantor set \cite{jorba2001fine}; consequently, from a practical perspective, realistic Solar System parameters are not sufficiently small to guarantee the survival of these invariant structures. In the true ephemeris model, the invariant tori of the CR3BP are destroyed, leaving behind only dynamical shadows of these structures \cite{gomez_dynamics_2001-3,gidea2007shadowing}, and rendering trajectories vulnerable to slow chaotic drift through the resonance web \cite{arnol2020instability,morbidelli2002modern}. This fundamental limitation of KAM theory in practical situations motivated the shift toward computing effective stability times \cite{celletti2007kam,jorba1998numerical}. 

Conversely, the concept of effective stability, pioneered by Nekhoroshev in \cite{nekhoroshev1977exponential}, focuses on bounding the drift of trajectories over exponentially long, but finite, time intervals. Following Nekhoroshev's original theorem, the rigorous bounds, optimal stability exponents, and geometric confinement conditions for continuous Hamiltonian flows were successively refined \cite{benettin1985proof,lochak1992canonical,poschel1993nekhoroshev}. These collective proofs established that for an analytic, nearly-integrable Hamiltonian system satisfying specific geometric conditions (such as steepness or quasi-convexity), the variation in the action variables remains strictly bounded for a time $T$ that scales exponentially with the inverse of the perturbation size $\epsilon$, i.e., $T \, \propto \, \exp(\epsilon^{-a})$. Meanwhile, effective stability estimates were proved for realistic systems in celestial mechanics, such as the region around the triangular points of the CR3BP \cite{giorgilli1989effective,benettin1999nekhoroshev}, with applications to the Sun-Jupiter system, motivated by the observation of the Trojan asteroids \cite{nicholson1961trojan}. While classical effective stability was developed for continuous Hamiltonian flows, extending these bounds to discrete symplectic mappings is essential for the present work, as periodic orbits are naturally studied as fixed points of a Poincar\'e map. Formulated by Turchetti \cite{turchetti1990nekhoroshev} and Bazzani \cite{bazzani1987normal,bazzani1989nekhoroshev,bazzani1991first}, then further developed by Efthymiopoulos \cite{efthymiopoulos2005formal,efthymiopoulos2005optimized,efthymiopoulos2008connection}, the foundational framework for discrete Nekhoroshev estimates demonstrates that in the discrete setting, the accumulation of remainder terms per map iteration drives the topological drift, requiring a careful algebraic construction to bound the escape time of the discrete trajectory. 

% I.A.3. Normal Forms and Long-Time Stability in Celestial Mechanics 
\subsubsection{Normal Forms and Long-Time Stability in Celestial Mechanics}
A critical distinction must be made when applying these classical stability theorems to cislunar astrodynamics. In celestial mechanics problems, such as the planetary $N$-body problem, the dynamical system is often formulated as globally nearly-integrable \cite{chierchia2011planetary,yalinewich2020nekhoroshev}; the dynamics are dominated by a massive central body (providing an integrable Keplerian part), with the other bodies treated as inherently small perturbations globally scaled by a mass parameter $\epsilon$. By contrast, the Earth-Moon CR3BP in the regime of the NRHOs is a strongly non-integrable, highly nonlinear system. The gravitational influence of the Moon during close perilune passages is not a small perturbation, and the system admits no global nearly-integrable decomposition \cite{Szebehely67}. 

However, while the global system is non-integrable, the dynamics strictly local to a stable periodic orbit can be constructed as a nearly-integrable system. By expanding the flow as a Taylor series centered on the fixed point of the corresponding Poincar\'e map, the linear dynamics of the center manifold (an integrable twist map) serve as the unperturbed, integrable part. The higher-order nonlinear terms of the expansion subsequently acts as the perturbation, intrinsically scaled not by a physical mass parameter, but by the spatial distance from the reference orbit \cite{gimeno2025explicit}. This conceptual pivot--synthesizing near-integrability from the local series expansion--allows the powerful machinery of KAM and Nekhoroshev theories to be applied to strongly globally non-integrable models in celestial mechanics. Moreover, this perspective has been illustrated by several authors, including Giorgilli et al. \cite{giorgilli1989effective} and Jorba \& Villanueva \cite{jorba1998numerical}, which particularly inspire the present work; yet, the methodology presented here builds off of Gimeno et al. \cite{gimeno2025explicit} and is more generic mathematically and more readily applied for computations in realistic systems than previous works. 

The primary tool for deriving effective stability bounds is normal form theory. By constructing a sequence of near-identity coordinate transformations, normal forms systematically eliminate certain perturbation terms up to a chosen truncation order, often (unless otherwise specified) isolating the completely integrable dynamics. Historically, in celestial mechanics, these transformations were computed using the method of Lie series or Lie transforms (e.g., Deprit \cite{deprit1969canonical}, Gustavson \cite{gustavson1966on}) \cite{giorgilli2022notes,Jorba99a}. However, executing these methods were limited at the time by massive symbolic algebraic manipulation, leading to exponential expression swell and restricting explicit computations to relatively low orders. 

A critical recent innovation in the computation of normal forms is the application of jet transport (sometimes referred to as Taylor differential algebra) \cite{gimeno2025explicit,gimeno2023numerical}. Rather than relying on symbolic manipulation or finite differencing, jet transport rigorously propagates the Taylor expansion of the flow along the reference trajectory using automatic differentiation. As demonstrated in Gimeno et al. \cite{gimeno2025explicit}, this allows for the explicit, high-order computation of the Poincar\'e map and its derivatives (and consequently its normal form) to arbitrary precision. This recent advancement enables the rigorous evaluation of the map's high-order remainder tail, shifting Nekhoroshev estimates from pen-and-paper theorems into semi-analytical computations for rapid applications to realistic and relevant systems.

% I.B. Main Questions and Summary of Results 
\subsection{Main Questions and Summary of Results}
The coexistence of linear stability and severe nonlinearity in the Earth-Moon NRHOs motivates the central question of this work: Given a macroscopic neighborhood around a normally elliptic periodic orbit, what is the guaranteed, unperturbed survival time of a trajectory before it escapes? Conversely, what is the effective stability region, and to what extent does it depend on the escape time? 

In this paper, we bridge rigorous dynamical systems theory with numerical cislunar astrodynamics to answer these questions explicitly. Our main contributions are as follows:
\begin{enumerate}
    \item[1.] \textbf{Discrete Effective Stability Bound:} We formulate and prove a rigorous effective stability theorem (Theorem 8) for an elliptic fixed point of a general (symplectic) Poincar\'e map. By optimizing the truncation order of the discrete normal form, we derive a strict lower bound on the survival time $T_{\text{eff}}(a)$ as a continuous function of the initial confinement radius $a$.
    \item[2.] \textbf{Physical Confinement Mechanism:} We demonstrate that, despite the exact symplectic structure being broken at the truncation order by explicit polynomial coordinate transformations, the trajectory remains rigorously confined by solely bounds on the orbital radii, independent of strictly canonical action variables. 
    \item[3.] \textbf{Application to Earth-Moon $L_2$ NRHOs:} Utilizing jet transport, we explicitly compute the high-order discrete normal forms for a periodic orbit in the stable band of Earth-Moon $L_2$ NRHOs. By evaluating the geometric bounds of our optimized normal form, we map the maximum analytical confinement radius directly to physical spatial domains. Because the theoretical stability times within this region vastly exceed practical mission lifetimes, these derived effective stability regions define the fundamental limits of unperturbed, autonomous confinement within the CR3BP.  
\end{enumerate}
Through these results, we establish a fully explicit, computable framework for bounding the nonlinear behavior of spacecraft near elliptic periodic orbits, providing a mathematically rigorous foundation for NRHO mission lifetimes and far beyond.

%%%%%%%%%%%%%%%%%%%%%%%%%%%%%%%%%%%%%%%%%%%%%%%%%%
% II. Preliminaries and Normal Form Construction %
%%%%%%%%%%%%%%%%%%%%%%%%%%%%%%%%%%%%%%%%%%%%%%%%%%
\section{Preliminaries and Normal Form Construction}\label{sec:Preliminaries}
In this section, we establish the framework required to analyze the discrete local dynamics in the macroscopic vicinity of a periodic orbit. To study the long-term behavior of nearby trajectories, we first reduce the continuous flow of the dynamical system to a discrete-time mapping via a suitable defined Poincar\'e section. Section~\ref{sec:PMJT} formally introduces this Poincar\'e map and details the application of jet transport to efficiently compute its high-order Taylor expansion around the fixed point, referring to earlier works of \cite{gimeno2023numerical,gimeno2025explicit} that serve as the basis of our analysis. Taking this local polynomial expansion as a starting point, Section~\ref{sec:Normalization} details the algebraic construction of the discrete normal form, outlining the sequence of near-identity coordinate transformations necessary to systematically eliminate non-integrable nonlinearities, and formally defining the truncation remainder. Finally, Section~\ref{sec:InvariantStructures} explores the geometric properties of the resulting normalized map, which is an integrable twist map, that serves as the foundation for the rigorous stability bounds derived in the subsequent sections. 

% II.A. The Poincar\'e Map and Jet Transport 
\subsection{The Poincar\'e Map and Jet Transport}\label{sec:PMJT}
% 1. Define P & the fixed point x_0 
Throughout Sections 2-4, we consider a general autonomous dynamical system governed by the ordinary differential equation $\dot{x} = f(x)$ with $x \in \mathbb{R}^{n_{\text{sys}}}$, which admits a periodic orbit. To analyze the local dynamics, we reduce the continuous flow to a discrete Poincar\'e map $P\colon \Sigma \to \Sigma$ defined on an $2n$-dimensional transversal surface of section, $\Sigma$, where the periodic orbit corresponds to a fixed point $x_0 = P(x_0)$. Note that we let $2n = n_{\text{sys}}-2$ because we restrict to a constant energy surface. While the Earth-Moon CR3BP (our motivating example) is a Hamiltonian system, the normalization and subsequent effective stability estimates derived in these sections do not require the explicit construction of a globally canonical Hamiltonian. Rather, we invoke the Hamiltonian structure of the underlying vector field $f(x)$ only insofar as it provides the spectral structure of the discrete linear map $\dop P(x_0)$; in particular, it guarantees that the eigenvalues of this Jacobian matrix come in reciprocal pairs $(\lambda_i,\lambda_i^{-1})$, which dictates the form of our linear normalization in the following section. 

% 2. Use jet transport to get the Taylor polynomial G(s) 
To explicitly construct the discrete normal form, the fundamental prerequisite is the high-order Taylor expansion of the Poincar\'e map around the fixed point $x_0$. In the classic celestial mechanics literature, constructing Hamiltonian normal forms around equilibrium points of flows, as in \cite{giorgilli1989effective,Jorba99a,celletti1990stability}, performs the expansion of the system Hamiltonian by Richardson \cite{richardson1980note}. As we do not require an expansion of the full system Hamiltonian, but rather a Taylor expansion of the dynamics on the Poincar\'e section, we use the method of jet transport to construct the series. Let $s \in \mathbb{R}^{2n}$ denote the local coordinates centered on the fixed point. Note that $s$ is sometimes referred to as a ``symbol'' in the literature. Our objective is to compute the polynomial representation $G(s) = P(x_0+s)$ truncated to a desired order $N$. 

Historically, extracting the high-order derivatives of a map derived from a continuous, strongly nonlinear flow (such as the CR3BP) presented a formidable computational bottleneck. Some authors have derived simplified discrete-time models of the CR3BP to avoid the computational challenges, such as S\'andor in \cite{sandor2002symplectic} and later used by Efthymiopolous in \cite{efthymiopoulos2005formal}. Finite-difference approximations can cause catastrophic cancellation and round-off errors at higher-derivatives. 

To overcome this limitation, we utilize the method of jet transport (sometimes referred to as Taylor differential algebra \cite{berz1998verified}). Rather than evaluating the dynamical system using standard floating-point arithmetic on discrete numerical states, jet transport redefines the arithmetic operations to act on the space of truncated Taylor polynomials (referred to as jets) \cite{gimeno2023numerical}. If two functions can be represented by their $N$-th order Taylor series, any algebraic operation or standard intrinsic function (e.g., sine, exponential, square root) applied to them yields a new $N$-th order polynomial, with the truncated algebra handled automatically and up to machine precision.

In practice, instead of integrating a single initial condition $x_0$, we initialize the numerical integration with the polynomial state $X_0 = x_0 + s$, where $s$ acts as an algebraic variable (a symbol) representing a continuous local neighborhood of $x_0$. We then propagate this polynomial state through the vector field $\dot{x} = f(x)$ using a high-order Taylor integrator--in this work, the Taylor package \cite{jorba2005software,Gimeno2022TaylorPackageVersion}. As demonstrated in earlier works developing jet transport for rigorous Taylor integration (e.g., Gimeno et al. \cite{gimeno2023numerical}), this numerical scheme naturally accommodates the polynomial algebra. The integrator step-size and order are controlled dynamically, propagating the entire neighborhood of trajectories. 

Because the Poincar\'e section $\Sigma$ is defined as a spatial section, the time of flight varies for initial conditions within the neighborhood of $x_0$. Hence, the crossing time is itself computed as an $N$-th order polynomial $T(s)$, as outlined in \cite{gimeno2025explicit}. By evaluating the flow at this polynomial time $\varphi_{T(s)}(x_0 + s)$, and projecting the result onto the transversal section (see \cite{gimeno2023numerical} for details), we directly obtain the Taylor expansion of the Poincar\'e map around the fixed point:
\begin{equation}
    P(x_0 + s) = G(s) = \sum_{k\geq 0} G^k(s) = \sum_{k \geq 0} \sum_{|j| = k} G_j^k s^j
\end{equation}
where each $G^k(s)$ is a homogeneous polynomial of degree $k$ in the local variables $s$, and where $j$ is a multi-index. Note that the translation of $x_0$ to the origin of $G(s)$ has been considered as the $0$-th order change of variables, $c_0(s)$, in prior works \cite{gimeno2025explicit}. Of course, in the numerical example of Section 5, the initial representation of $G(s)$ is taken as the $N$-th order truncation. This approach provides the Taylor expansion of the Poincar\'e map, yielding the precise algebraic foundation required for the subsequent normal form construction.  

% II.B. Normal Forms on the Poincar\'e Map Near an Elliptic Fixed Point 
\subsection{Normal Forms for the Poincar\'e Map Near an Elliptic Fixed Point}\label{sec:Normalization}
With the high-order Taylor polynomial of the Poincar\'e map, $G(s)$, explicitly computed via jet transport, we now turn to the algebraic simplication of the local dynamics. The fundamental objective of discrete normal form theory is to construct a sequence of near-identity coordinate transformations that systematically strips away the non-integrable, cross-coupling terms up to a specified truncation order. In sufficiently smooth systems, the Poincar\'e-Dulac Theorem, \cite{arnold2012geometrical} proves that, locally, any discrete system is formally equivalent to a formal discrete dynamical system only containing resonant monomials \cite{gimeno2025explicit}. By applying this change of variables, the map is transformed to a simplified normal form of an integrable twist map, i.e., solely amplitude dependent, pushing the chaotic, angle-dependent perturbations into a high-order remainder tail. We construct this transformation in two phases: first, by diagonalizing the linear part of the map to establish a complex eigenbasis, and second, by recursively solving the homological equation (derived below) to normalize the higher-order terms. 

\subsubsection{Linear Normalization} 
The first step in the normalization process is to simplify the linear part of the map, given by the Jacobian matrix $G^1$ evaluated at the fixed point $x_0$. Because we are considering a fully elliptic fixed point of a system with a symplectic linear spectrum, the eigenvalues of $G^1$ lie on the complex unit circle in the complex plane, and, crucially, strictly occur in complex conjugate (reciprocal) pairs, $\lambda_i = e^{\imag\omega_i}$ and $\bar{\lambda}_i = \lambda_i^{-1} = e^{-\imag\omega_i}$ for $i = 1,\ldots,n$. These eigenvalues represent the fundamental rotational frequencies of the linear center manifold. 

% 1. Diagonalization 
We diagonalize the linear dynamics using the eigenvector decomposition of the Jacobian matrix. Let $V$ and $\Lambda$ be matrices over $\mathbb{C}$ representing the eigenvectors and eigenvalues of $G^1$, respectively, such that:
\begin{equation}
    \Lambda = V^{-1} G^1 V = \text{diag}\left(\lambda_1,\cdots,\lambda_n,\lambda_1^{-1},\cdots,\lambda_n^{-1}\right)
\end{equation}
Utilizing this decomposition, we define a linear change of variables by $c_1(s) = (\mathfrak{p} V)s$, which effectively transforms the real, physical local coordinates $s \in \mathbb{R}^{2n}$ into a set of complex conjugate coordinates $c_1(s) = (z,\bar{z})$, where $z = (z_1,\ldots,z_n) \in \mathbb{C}^n$. As established in the standard normal form framework for symplectic maps (e.g., Turchetti \cite{turchetti1990nekhoroshev} and Bazzani \cite{bazzani1987normal}), passing to these complex coordinates is algebraically essential. In this basis, the linear part of the map acts by strictly independent phase rotations, $z_i \mapsto e^{\imag\omega_i}z_i$ and $\bar{z}_i \mapsto e^{-\imag\omega_i}\bar{z}_i$. Consequently, any complex monomial forms by these variables is an eigenfunction of the linear operator, a property that perfectly diagonalizes the homological equations encountered in the subsequent higher-order normalization steps. 

Applying this transformation to the full polynomial map yields the first-order normal form, $F^{(1)}$, defined as:
\begin{equation}
    F^{(1)} = c_1^{-1} \circ G \circ c_1
\end{equation}
By construction, the linear part of $F^{(1)}$ is exactly the diagonal matrix $\Lambda$. The higher-order terms of $F^{(1)}$, i.e., $|j| = k \geq 2$, now represent the non-integrable nonlinearities expressed in the complex coordinate basis (sometimes referred to as the diagonal coordinates). Note that the scaling factor $\mathfrak{p}$, a parameter used to improve numerical performance, has no effect on the estimates derived in the following sections, as it is canceled through conjugations; see \cite{gimeno2025explicit} for details. 

\subsubsection{Finite-Order Normal Form Construction}
The objective of the finite-order normal form is to construct a sequence of near-identity coordinate transformations, starting with the linearly diagonalized map $F^{(1)}$, that iteratively eliminates non-resonant nonlinearities up to a desired truncation order $N$. After applying $N-1$ iterative normalization steps (plus the first normalizing step done by diagonalization), the resulting map takes the form 
\begin{equation}
    P(x_0+s) = F^{(N)}(s) + R^{(N)}(s) = \Lambda s + \sum_{k=2}^N F^{k,(N)}(s) + \mathcal{O}_{N+1}
\end{equation}
where $F^{k,(N)}(s)$ represents the normalized homogeneous polynomials of degree $k$, and $R^{(N)}(s)$ is the truncation remainder. 

To construct this normal form, we proceed inductively. Assume the map has been normalized up to degree $k-1$, yielding $F^{(k-1)}(s)$. To normalize the degree-$k$ terms, we introduce a near-identity coordinate transformation
\begin{equation}
    c_k(s) = s - \chi_k(s) 
\end{equation}
where $\chi_k(s) = \sum_{|j|=k} b_js^j$ is an unknown homogeneous polynomial of degree $k$ with vector coefficients $b_j \in \mathbb{C}^{2n}$. The inverse of this transformation, expanded to order $k$, is simply $c_k^{-1}(s) = s + \chi_k(s) + \mathcal{O}_{2k-1}$. Note that the choice of sign $+/-$ in the near-identity transformation is arbitrary and has been chosen as such in the present work to match \cite{gimeno2025explicit}.

We apply this transformation via conjugation to obtain the next step in our normal form
\begin{equation}
    F^{(k)}(s) = c_k^{-1} \circ F^{(k-1)} \circ c_k(s)
\end{equation}
To explicitly see how the transformation modifies the degree-$k$ terms, let $y = F^{(k-1)}(c_k(s))$. Since $\chi_k(s)$ contains only terms of degree $k$, evaluating the Taylor expansion of $F^{(k-1)}$ at $s - \chi_k(s)$ yields
\begin{equation}
    y = \Lambda (s - \chi_k(s)) + \sum_{m=2}^{k} F^{m,(k-1)}(s) + \mathcal{O}_{k+1}
\end{equation}
We then apply the inverse transformation, $F^{(k)}(s) = y + \chi_k(y)$. Because $\chi_k$ is a homogeneous polynomial of degree $k$, evaluating it at the linearly dominant argument $(\Lambda s + \mathcal{O}_2)$ simplifies to $\chi_k(\Lambda s) + \mathcal{O}_{k+1}$. Collecting all terms up to degree $k$, the composition becomes
\begin{equation}
    F^{(k)}(s) = \Lambda s + \sum_{m=2}^{(k-1)} F^{m,(k-1)} (s) + \left[ R^{k,(k-1)}(s) + \chi_k(\Lambda s) - \Lambda \chi_k(s) \right] + \mathcal{O}_{k+1}.
\end{equation}
This expansion demonstrates that the lower-order normalized terms $(m < k)$ remain undisturbed. At degree $k$, the new terms $F^{k,(k)}(s)$ are governed by the relation
\begin{equation}
    F^{k,(k)}(s) = R^{k,(k-1)}(s) + \chi_k(\Lambda s) - \Lambda \chi_k(s)
\end{equation}
Because the linear map acts diagonally, i.e., $\Lambda s = (\lambda_1 s_1, \ldots, \lambda_{2n} s_{2n})$, the complex monomials are eigenfunctions of $\Lambda$. Thus, $\chi_k(\Lambda s) = \sum b_j \lambda^j s^j$, where $\lambda^j = \lambda_1^{j_1} \cdots \lambda_{2n}^{j_{2n}}$. Now, looking at the $i$-th component of a specific multi-index $j$ (with $|j|=k$), we have:
\begin{equation}
    F_{j,i}^{k,(k)} = R_{j,i}^{k,(k-1)} + b_{j,i} \lambda^j - \lambda_i b_{j,i}
\end{equation}
Rearranging this provides the coefficient-level homological equation for discrete maps
\begin{equation}
    (\lambda_i - \lambda^j) b_{j,i} = R_{j,i}^{k,(k-1)} - F_{j,i}^{k,(k)}
\end{equation}

The goal of the normalization step is to eliminate the nonlinear coupling terms by forcing $F_{j,i}^{(k)} = 0$. When this is possible, the required transformation coefficient is given by the scalar division
\begin{equation}
    b_{j,i} = \frac{R_{j,i}^{k,(k-1)}}{\lambda_i - \lambda^j}
\end{equation}
However, this elimination is only possible if $\lambda_i - \lambda^j \neq 0$. A resonance thus occurs whenever the divisor vanishes. If $\lambda_i - \lambda^j = 0$, the coefficient $b_{j,i}$ is undefined, meaning we cannot eliminate the term. Instead, we must set $b_{j,i} = 0$ and retain the resonant term in the normal form, such that $F_{j,i}^{(k)} = F_{j,i}^{(k-1)}$ for the particular multi-index $j$ and coordinate $i$. 

Crucially, because the linear spectrum of an elliptic fixed point in a Hamiltonian system consists of reciprocal complex conjugate pairs ($\lambda_i = e^{\imag\omega_i}, \lambda_{i+n} = e^{-\imag\omega_i}$), trivial resonances are unavoidable at all odd degrees $(k = 3,5,\ldots)$. For example, terms containing combinations like $s_i(s_m s_{m+n}) = s_i |s_m|^2$ trivially satisfy $\lambda_i - (\lambda_i \lambda_m \lambda_m^{-1}) = 0$. These unavoidable resonant terms must be kept in the normal form, and they are precisely the terms responsible for the amplitude-dependent frequency shifts characterizing the integrable center manifold of the twist map. 

Finally, for the non-resonant terms where $\lambda_i - \lambda^j \neq 0$, simply having a non-zero divisor is insufficient to guarantee stability. To prevent the coefficients of $c_k$ from growing uncontrollably due to arbitrarily small divisors--a central challenge in deriving Nekhoroshev estimates--we require the fundamental frequencies to satisfy a strict Diophantine condition. Specifically, there must exist real constants $\gamma > 0$ and $\tau \geq n$ such that for all non-resonant combinations
\begin{equation}
    |\lambda_i - \lambda^j| \geq \frac{\gamma}{k^\tau}, \qquad |j| = k 
\end{equation}
By systematically applying the homological equation subject to this Diophantine condition, we construct the discrete normal form order-by-order up to degree $N$, formally preparing the system for the rigorous remainder bounds of Sections 3 and 4. 

\subsubsection{Canonicity of Transformations}
A natural question arises regarding the geometric properties of the near-identity transformations $c_k(s)$ constructed in the normalization procedure. Those accustomed to classical Hamiltonian perturbation theory may observe that the direct polynomial transformations employed here are not strictly canonical (symplectic). It is important then to clarify why this relaxation of exact symplecticity is both computationally advantageous and analytically permissible within the context of deriving effective stability bounds, as is our goal. 

In the classical Hamiltonian normal form constructions for continuous-time settings, normal forms are often constructed using Lie series or Lie transforms \cite{giorgilli2022notes}. In that approach, the coordinate transformation is defined as the time-1 flow of a generating Hamiltonian function. Because the transformation is a Hamiltonian flow, it is strictly canonical by construction, and the symplectic structure is perfectly preserved at all orders up to truncation. 

When extending these concepts to discrete symplectic maps, maintaining symplecticity becomes significantly more cumbersome. To rigorously preserve the symplectic structure, standard approaches--such as those developed by Bazzani, Guzzo, and others \cite{bazzani1987normal,guzzo2004direct}--often rely on discrete generating functions of mixed variables (e.g., $S(q,P) = qP + W(q,P)$). While this implicitly guarantees that the resulting transformation is canonical, it requires solving implicit equations to recover the explicit map $(q,p) \mapsto (Q,P)$, a process that is computationally prohibitive and algebraically dense when evaluating high-order expansions via jet transport. 

In contrast, the direct polynomial conjugation method we employ, $F^{(k)} = c_k^{-1} \circ F^{(k-1)} \circ c_k$, relies on explicit, finite-order polynomial inversions. Consequently, the transformation $c_k(s) = s - \chi_k(s)$ is symplectic up to the normalization order $k$; moreover, using polynomial mappings (and, in particular, their inversions) inherently introduces artificial violations of the symplectic condition at $\mathcal{O}_{k+1}$. 

To explicitly guarantee exact symplecticity without solving implicit equations, one could instead represent the transformations using Lie operators via Dragt-Finn factorization \cite{dragt1976lie,koseleff1994formal}. However, evaluating a Lie exponential yields an infinite series. Because our jet transport methodology inherently requires finite-degree polynomial truncation for computability (as does any computational method), truncating the Lie series would re-introduce symplectic violations at $\mathcal{O}_{N+1}$. In fact, this is the situation for Hamiltonian normal forms in the continuous-time setting: the transformations constructed to be canonical break canonicity at $\mathcal{O}_{N+1}$. Therefore, while explicit, finite-order polynomial normalization inherently precludes exact symplecticity, relaxing the strict canonical requirement provides an effective bridge between rigorous analytical stability bounds and practical numerical computation. 

Furthermore, this lack of exact canonicity does not invalidate the subsequent stability results. The fundamental architecture of Nekhoroshev estimates and effective stability theory relies strictly on analytic bounding techniques to control the exponential growth of the remainder norm $\|R^{(N)}\|$. While the underlying symplectic geometry is responsible for the existence of the paired eigenvalues and the invariant twist map structure, the quantitative bounds on the escape time of trajectories depend exclusively on the magnitude of the remainder pushing the trajectory off of those structures. Because our explicit polynomial transformations successfully push the non-integrable terms into a bounded $\mathcal{O}_{N+1}$ remainder, the analytic requirements for the Nekhoroshev estimates are fully satisfied, rendering the exact preservation of the symplectic 2-form unnecessary for the final stability bounds. 

In summary, while a classical Lie-based factorization would theoretically preserve symplecticity and eliminate artificial dissipation, computing the required nested Poisson brackets induces severe combinatorial growth in the majorant bounds (Section 3), ultimately leading to more conservative Nekhoroshev estimates (Section 4). Conversely, the discrete generating function approaches of Bazzani, Turchetti, and others successfully avoid this combinatorial explosion while preserving symplecticity; however, they require the inversion of implicit, mixed-variable equations. This implicit inversion is computationally prohibitive when operating on high-order maps generated via jet transport from continuous flows. By relaxing the strict requirement of exact symplecticity, our method of direct explicit polynomial composition occupies a methodological sweet spot. Because effective stability theorems bound all un-normalized dynamics, absorbing both Hamiltonian chaos and artificial non-canonical dissipation into the $\mathcal{O}_{N+1}$ remainder, direct polynomial composition is completely permissible. Consequently, our approach guarantees the computational feasibility absent in the Bazzani-Turchetti methods, while simultaneously yielding tighter analytical Cauchy estimates than Lie-based methods, all without compromising the rigorous bounds of the trajectory's escape time.

\subsubsection{Definition of the Remainder}
% Definition 
Let $c = c_1 \circ c_2 \circ \cdots \circ c_N$ denote the complete near-identity transformation up to a chosen truncation order $N$. Applying this full transformation to the Taylor expansion of the original Poincar\'e map centered at the fixed point $x_0$, i.e., $G(s) = P(x_0 + s)$, yields the transformed dynamics, which we split into a finite polynomial part (the normal form) and an infinite tail (the remainder)
\begin{equation}
    c^{-1} \circ G \circ c(s) = F^{(N)}(s) + R^{(N)}(s)
\end{equation}
Here, $F^{(N)}(s)$ is the completely normalized polynomial map of degree $N$ containing the linear rotations and the amplitude-dependent twist dynamics; in other words, $F^{(N)}$ is an integrable twist map. The function $R^{(N)}(s) = \mathcal{O}_{N+1}$ is the truncation remainder. 

% Analytical perspective 
From an analytical perspective, the remainder is the central object of study for the effective stability bounds derived in Section 3. Because $F^{(N)}(s)$ has invariant structures (i.e., tori), any drift of a trajectory away from the fixed point is driven exclusively by the remainder, $R^{(N)}(s)$. This remainder encapsulates all the physical, high-order non-integrable resonant couplings of the true system, as well as the artificial dissipative drift introduced by our non-canonical polynomial inversions. Bounding the norm of the remainder, $\|R^{(N)}\|$, over a local domain is the fundamental requirement for establishing the Nekhoroshev escape times. 

% Computational perspective 
From a computational perspective, evaluating an infinite series is practically impossible. Within our jet transport framework, the map and its transformations are computed as finite Taylor polynomials. Consequently, we cannot construct the infinite tail $R^{(N)}(s)$ explicitly term-by-term. Instead, we treat the remainder analytically using Cauchy estimates (Section 3). Because the original map $G$ and the transformations $c$ are analytic in a macroscopic complex neighborhood of the origin (a fixed point of $G$), the remainder is also an analytic function. Rather than computing the exact coefficients of $R^{(N)}(s)$, we utilize the bounds on the finite-order transformations and the domain of analyticity of $G$ to strictly bound the norm of the infinite tail over a polydisc. Thus, the computational task shifts from tracking an infinite number of coefficients to rigorously computing the majorant bounds of the finite operations that generated them. 

% II.C. Invariant Structures of the Truncated Normal Form
\subsection{Invariant Structures of the Truncated Normal Form}\label{sec:InvariantStructures}
Having constructed the Birkhoff normal form $F^{(N)}$ up to order $N$, we now examine the geometric and dynamic properties of this truncated system. Because the homological equations were solved under a non-resonant Diophantine condition, all angle-dependent coupling terms up to order $N$ have been eliminated. The only nonlinearities that survive the normalization procedure are the unavoidable ones, i.e., terms composed exclusively of combinations like $z_i(z_m \bar{z}_m)^{j_m}$. 

To interpret the dynamics of $F^{(N)}$ physically, it is natural to introduce a set of effective action variables, defined component-wise as
\begin{equation}
    I_i = z_i \bar{z}_i, \quad \text{for } i = 1,\ldots,n
\end{equation}
In classical Hamiltonian mechanics, these variables represent the unperturbed actions of the linear center manifold. Although our explicit polynomial transformations are not strictly canonical, these variables still perfectly capture the invariants of the truncated normalized dynamics. When expressed in terms of these actions, the truncated map $z' = F^{(N)}(z,\bar{z})$ factorizes into the integrable twist map
\begin{equation}
    z_i' = z_i \exp \left( \imag \Omega_i^{(N)}(I) \right) 
\end{equation}
Here, $\Omega_i^{(N)}(I)$ represents the nonlinear frequency of the $i$-th oscillator, which expands naturally as a polynomial in the action vector $I = (I_1,\ldots,I_n)$
\begin{equation}
    \Omega_i^{(N)}(I) = \omega_i + \sum_{1 \leq |m| \leq \lfloor \frac{N-1}{2} \rfloor} \beta_{i,m} I^m
\end{equation}
where $\omega_i$ are the base linear frequencies obtained from the initial Jacobian diagonalization, $m \in \mathbb{Z}_+^n$ is a multi-index, and $\beta_{i,m}$ are real coefficients capturing the nonlinear twist induced by the unavoidable resonances. 

Under this truncated map, computing the magnitude of the updated coordinates reveals that the action variables are strictly conserved
\begin{equation}
    I_i' = \frac{1}{2} z_i' \bar{z}_i' = \frac{1}{2} \left( z_i e^{\imag\Omega_i^{(N)}(I)} \right) \left( \bar{z}_i e^{-\imag\Omega_i^{(N)}(I)} \right) = \frac{1}{2} z_i \bar{z}_i = I_i
\end{equation}
Because the actions $I_i$ are constants of motion for the map $F^{(N)}$, the phase space of the truncated system is foliated by a continuous family of invariant tori, parameterized by the constant vector $I = I^{(0)}$ (or, equivalently, by the frequencies). Geometrically, an initial condition placed on a specific torus $I^{(0)}$ will remain on that exact torus for infinite time (perpetual stability) under the iteration of $F^{(N)}$, acting as a rigid quasi-periodic rotation at the constant frequency $\Omega_i^{(N)}(I^{(0)})$. These unperturbed invariant tori define the foundational topological skeleton of the local center manifold, and may be similarly used to construct local orbital elements for astrodynamics applications, as in \cite{peterson2023local,peterson2024gauss,peterson2024local}. 

While the action variables provide the most physically intuitive description of the integrable twist map, deriving rigorous effective stability estimates requires bounding the complex analytic extensions of the remainder $R^{(N)}$. For these analytical bounds, it is a bit cleaner to measure phase space domains using the geometric orbital radii $r = (r_1,\ldots,r_n) \in \mathbb{R}_+^n$, defined by their squares 
\begin{equation} 
    r_i^2 = z_i \bar{z}_i = 2 I_i
\end{equation}
Following the convention of Gimeno et al. \cite{gimeno2025explicit}, substituting the physical actions with these squared radii allows us to directly parameterize the invariant tori by their Euclidean amplitude in the complex eigenbasis (i.e., the diagonal coordinates). The twist frequencies can be equivalently expressed as $\Omega_i^{(N)}(r^2)$, and the tori remain defined by constant radii $r = r^{(0)}$. In the subsequent sections, we transition to using these orbital radii to formulate the main stability theorem, as they provide a direct metric for bounding the chaotic drift induced when the full non-integrable remainder is re-introduced to the system.

%%%%%%%%%%%%%%%%%%%%%%%%%%%%%%%%%%%%%%%%%%%%%%%%%%%%%%
% III. Remainder Estimates and Optimal Normalization %
%%%%%%%%%%%%%%%%%%%%%%%%%%%%%%%%%%%%%%%%%%%%%%%%%%%%%%
\section{Remainder Estimates and Optimal Normalization}

To rigorously prove that trajectories in the vicinity of an elliptic fixed point remain confined for long durations, we must bound the truncation remainder of the normal form. As established in Section 2, the sequence of near-identity transformations simplifies the nonlinear dynamics but unavoidably generates small divisors. These divisors cause the coefficients of the normal form to grow factorially, meaning the infinite series is generally divergent. To extract meaningful finite-time stability guarantees, we must therefore truncate the series at a finite order $N$. This introduces a fundamental competition: increasing $N$ pushes the un-normalized remainder to a higher-polynomial degree, making it geometrically small near the origin, but the factorial growth of the coefficients eventually overpowers this geometric decay. 

This section is dedicated strictly to quantifying this competition by deriving the analytical bounds for the normal form transformations and their effect on the remainder. We proceed by first establishing Cauchy estimates for the initial map, then bounding the coefficient growth during the normalization iterations, and finally quantifying the remainder after truncation. This sequence culminates in the derivation of the optimal normalization order $N_{\text{opt}}$ that minimizes the remainder for a given domain radius. The translation of these purely algebraic bounds into physical propositions relating to the action drift, stability time, and confinement region, is reserved for Section 4. 

To streamline the subsequent proofs and ensure consistency across the iterative estimates, we first establish the domain of analyticity and the norms utilized throughout this analysis. 

\subsubsection{Domain and Radius of Convergence} 
Consider a local, real-analytic, symplectic return map $G(s)$ near an elliptic fixed point at the origin. We assume that the initial Taylor series of $G(s)$ has a strictly positive radius of convergence. Specifically, there exists a real radius $\rho > 0$ such that the expansion converges absolutely on the complex polydisc centered at the origin, defined as 
\begin{equation}
    \mathcal{D}_\rho = \{ s \in \mathbb{C}^{2n} : |s_i| < \rho, \, i = 1,\ldots,2n \}.
\end{equation}
This radius of convergence $\rho$ establishes the maximum physical size of the analytical domain upon which our Taylor coefficients are well-defined and bounded. 

\subsubsection{Norms}
On the polydisc $\mathcal{D}_\rho$, we measure the magnitude of a scalar homogeneous polynomial of degree $k$ in $2n$ variables, $P_k(s) = \sum_{|j|=k} c_j s^j$, using the polynomial 1-norm:
\begin{equation}
    \|P_k\|_\rho = \sum_{|j|=k} |c_j| \rho^k
\end{equation}
We can isolate the sum of the absolute values of the coefficients by defining $C_k \coloneqq \sum_{|j|=k}|c_j|$, yielding $\|P_k\|_\rho = C_k \rho^k$. Note that the norm is sub-multiplicative, i.e., $\|P \cdot Q\|_\rho \leq \|P\|_\rho \|Q\|_\rho$. This property simplifies the bounding of the nested polynomial compositions that naturally arise when iterating normalizing transformations. 

For a vector-valued polynomial map $F(s) = (F_1(s),\ldots,F_{2n}(s))$ whose components are homogeneous polynomials of degree $k$, the norm is defined as the maximum over its components:
\begin{equation}
    \|F\|_\rho = \max_{i \in \{1,\ldots,2n\}} \|F_i\|_\rho 
\end{equation}

Finally, when evaluating linear operators (matrices) associated with these coordinate transformations, we employ the consistent induced matrix operator norm. For the supremum vector norm defined above, this corresponds precisely to the maximum row sum norm:
\begin{equation}
    \|A\| = \max_{i \in \{1,\ldots,2n\}} \sum_{m} |A_{i,m}|
\end{equation}

This framework guarantees that the norms of the initial Taylor expansion, the coordinate maps, their Jacobians, and the resulting nonlinear remainders remain compatible throughout the recursive estimates. 

% III.A. Bounds on the Normal Form Remainder 
\subsection{Bounds on the Normal Form Remainder}
With the analyticity domain and polynomial norms established, we are positioned to quantify the size of the coefficients for the initial Poincar\'e map.

\subsubsection{Cauchy Estimates for the Initial Map}
\begin{lemma}[Cauchy Estimates]
Assume that $G(s)$ is analytic on the complex polydisc $\mathcal{D}_\rho \coloneqq \{ |s|_\infty < \rho \}$. Let $C_k \coloneqq \sum_{|j|=k} |G_j^k|$. Then there exists computable constants $\tilde{C} > 0$ and $\tilde{\rho} > 0$ such that
\begin{equation}
    C_k \leq \frac{\tilde{C}}{\tilde{\rho}^k}.
\end{equation}
\end{lemma}

\begin{proof}
For analytic maps on $\mathcal{D}_\rho$, the standard Cauchy estimate gives:
\begin{equation}
    |G_j^k| \leq \frac{1}{\rho^k} \sup_{|s|_\infty < \rho} |G(s)|.
\end{equation}
Let $M$ be the constant
\begin{equation}
    M \coloneqq \sup_{|s|_\infty < \rho} |G(s)|, 
\end{equation}
so that the previous inequality is written more simply as
\begin{equation}
    |G_j^k| \leq \frac{M}{\rho^k}.
\end{equation}
Now, we bound the $C_k$ by
\begin{equation}
    C_k \leq \sum_{|j|=k} \frac{M}{\rho^k} = \frac{M}{\rho^k} \binom{n+k-1}{k},
\end{equation}
where the binomial coefficient comes from the number of $j$ such that $|j|=k$ in dimension $n$. 

To bound the binomial coefficient, consider
\begin{equation}
    \binom{n+k-1}{k} = \frac{(k+n-1)!}{k! (n-1)!} = \frac{1}{(n-1)!} \prod_{m=1}^{n-1} (k+m).
\end{equation}
For all $k \geq 1$ and all $1 \leq m \leq n-1$,
\begin{equation}
    k + m \leq k + (n-1),
\end{equation}
so that
\begin{equation}
    \prod_{m=1}^{n-1} (k+m) \leq (k+n-1)^{n-1}.
\end{equation}
Since $k \geq 1$, we have that $(n-1) \leq (n-1)k$ and thus (adding $k$ to both sides) $k+n-1 \leq nk$. Substituting this into the previous bound:
\begin{equation}
    \binom{n+k-1}{k} \leq \frac{n^{n-1}}{(n-1)!} k^{n-1} \eqqcolon b \cdot k^{n-1}. 
\end{equation}
Thus, substituting the bound on the binomial coefficient, we obtain the new bound for $C_k$ as
\begin{equation}
    C_k \leq \frac{M \cdot b}{\rho^k} k^{n-1}.
\end{equation}

Next, we need to remove the polynomial growth of the $k^{n-1}$, a common Nekhoroshev trick. Call $C = M \cdot b$, so that we have
\begin{equation}
    C_k \leq \frac{C}{\rho^k} k^{n-1}.
\end{equation}
We want a bound of the form
\begin{equation}
    C_k \leq \frac{\tilde{C}}{\rho^k},
\end{equation}
with explicit constants. The trick will be to replace $k^{n-1}$ by some exponential in $k$ and slightly shrinking the analyticity radius. Hence, we find a constant for ``exponential domination.'' Fix $\epsilon > 0$. Consider the function for $x > 0$:
\begin{equation}
    f(x) \coloneqq x^{n-1} e^{-\epsilon x}.
\end{equation}
This function $f$ is smooth and has a global maximum since $f \to 0$ as $x \to \infty$. Taking its derivative, we get
\begin{equation}
    f'(x) = x^{n-1}e^{-\epsilon x} \left( (n-1) - \epsilon x \right). 
\end{equation}
So, the critical point is: $x_* = (n-1)/\epsilon$. And the maximum value occurs at $x_*$ and is
\begin{equation}
    f(x_*) = \left( \frac{n-1}{\epsilon e} \right)^{n-1} e^{\epsilon x}. 
\end{equation}
The discrete version for integers $k \geq 1$ gives the bound
\begin{equation}
    k^{n-1} \leq \left( \frac{n-1}{\epsilon e} \right)^{n-1} e^{\epsilon k}. 
\end{equation}
Hence, we use this bound on the polynomial growth to obtain a uniform bound on the $C_k$ via:
\begin{equation}
    C_k \leq \frac{C}{R^k} \leq \frac{C}{R^k} \left( \frac{n-1}{\epsilon e} \right)^{n-1} e^{\epsilon k}.
\end{equation}
Denote by $\tilde{C}$ the new constant term
\begin{equation}
    \tilde{C} \coloneqq C \cdot (n-1/\epsilon e)^{n-1}. 
\end{equation} 
Convert the remaining exponential into a slightly smaller radius
\begin{equation}
    \frac{1}{\rho^k} e^{\epsilon k} = \frac{1}{(\rho e^{-\epsilon})^k}.
\end{equation}
Define $\tilde{\rho} \coloneqq \rho e^{-\epsilon} < \rho$, and note that $\tilde{\rho} \to \rho$ as $\epsilon \to 0$. Finally, we obtain the desired Cauchy estimate
\begin{equation}
    C_k \leq \frac{\tilde{C}}{\tilde{\rho}^k}. \label{eqn:CauchyEstimate}
\end{equation}
\end{proof}
With Equation \ref{eqn:CauchyEstimate}, the proof of Lemma 1 is complete. We have successfully established a rigorous, strictly geometric upper bound on the magnitude of the initial Taylor coefficients over the analyticity domain. This initial bound serves as the foundational baseline; in the following section, we will track exactly how these coefficients amplify as the normalization procedure begins.

\subsubsection{Coefficient Growth During Normalization}
Having bounded the original nonlinearities of the map, we must now quantify the accumulation of error and the growth of the coefficients as we iteratively put the system into normal form. At each step $k$, the map is updated by a near-identity coordinate transformation designed to eliminate the non-resonant terms of degree $k$. Mathematically, this normalization is done by successive conjugation
\begin{equation}
    G^{(k)} = c_k^{-1} \circ G^{(k-1)} \circ c_k, 
\end{equation}
where the coordinate transformation is defined as $c_k(s) = s - \chi_k(s)$. The generating function $\chi_k(s)$ is constructed by solving the homological equation:
\begin{equation}
    \chi_k(s) \coloneqq \sum_{|j|=k} \frac{R_{j,i}^{k,(k-1)}}{\lambda_i - \lambda^j} s^j, \qquad |\lambda_i - \lambda^j| \geq \frac{\gamma}{k^\tau}, \qquad \tau \geq 2, 
\end{equation}
where $\gamma$ and $\tau$ are positive constants defining the Diophantine condition on the frequencies. While this transformation successfully simplifies the $k$-th order dynamics, it inherently introduces two sources of analytical growth. First, the required division in the coefficients of $\chi_k(s)$ introduces small divisors, which directly amplify the magnitude of the generating function according to the Diophantine condition. Second, the explicit polynomial composition spawns a cascade of new, higher-order resonant and non-resonant terms (often referred to as the deformation or ``spillover''). 

To guarantee that the final remainder can be meaningfully bounded after $N$ steps, we must carefully control this growth. In this section, we derive iterative bounds that trace the worst-case accumulation of these coefficients through sequential conjugations. To manage the polynomial deformations algebraically, we employ a standard Cauchy technique: at each normalization step, we slightly shrink the radius of the analytical domain by a small amount $\delta$, ensuring the transformed map remains strictly bounded on the interior. 

First, we need to understand the coordinate transformations, finding appropriate bounds on the function, its derivative, and also the inverse map. These results are summarized and proved in the following lemma. 
\begin{lemma}[Coordinate Transformation]
    Let $\rho > 0$ be the radius of the polydisc $\mathcal{D}_\rho \subset \mathbb{C}^n$. Let $G^{k,(k-1)}(s)$ be the homogeneous polynomial of exact degree $k$ representing the un-normalized terms after $k-1$ steps. Assume the linear frequencies satisfy the Diophantine condition $|\lambda_i - \lambda^j| \geq \frac{\gamma}{k^\tau}$ for $\tau \geq 2$ and for all $|j|=k$. 

    Let $\delta = \frac{\rho}{6N}$ be the domain shrinking factor. Then, the homogeneous polynomial $\chi_k(s)$ of degree $k$ that solves the homological equation satisfies:
    \begin{enumerate}
        \item[1.] \textbf{Function bound:} $\|\chi_k\|_\rho \leq \frac{k^\tau}{\gamma} \|R^{k,(k-1)}\|_\rho$;
        \item[2.] \textbf{Derivative bound (Cauchy Estimate):} $\|\dop\chi_k\|_{\rho-\delta} \leq \frac{k^{\tau+1}}{\rho \cdot \gamma} \|R^{k,(k-1)}\|_\rho.$
    \end{enumerate}
    Moreover, if the un-normalized terms are small enough to satisfy the threshold condition
    \begin{equation}
        \frac{k^{\tau}}{\delta \gamma} \|R^{k,(k-1)}\|_\rho \leq \frac{1}{4},
    \end{equation}
    then, the coordinate transformation $c_k(s) = s -\chi_k(s)$ satisfies the following diffeomorphism guarantees:
    \begin{enumerate}
        \item[3.] \textbf{Forward Map:} $c_k$ is a biholomorphism onto its image, mapping $\mathcal{D}_{\rho - \delta} \to \mathcal{D}_\rho$; 
        \item[4.] \textbf{Inverse Map:} The inverse transformation $c_k^{-1}$ is rigorously defined as an analytic mapping $c_k^{-1}: \mathcal{D}_{\rho-2\delta} \to \mathcal{D}_{\rho-\delta}$.
    \end{enumerate}
\end{lemma}
\begin{proof}
    To prove (1), we simply apply the polynomial norm and the Diophantine condition:
    \begin{equation}
        \|\chi_k\|_\rho = \sum_{|j|=k} \left| \frac{R_j^{k,(k-1)}}{\lambda_i - \lambda^j} \right| \rho^k \leq \frac{k^\tau}{\gamma} \sum_{|j|=k} |R_j^{k,(k-1)}| \rho^k = \frac{k^\tau}{\gamma} \|R^{k,(k-1)}\|_\rho.  
    \end{equation}
    
    To prove (2), observe that the $i$-th component of $\chi_k$ is 
    \begin{equation}
        \chi_{k,i}(s) = \sum_{|j|=k} b_{j,i} s_1^{j_1} \cdots s_n^{j_n}.
    \end{equation}
    Taking the formal partial derivative with respect to $s_\ell$ brings down the exponent $j_\ell$:
    \begin{equation}
        \frac{\partial \chi_{k,i}}{\partial s_\ell} = \sum_{|j|=k, \, j_\ell \geq 1} b_{j,i} j_\ell s_1^{j_1} \cdots s_\ell^{j_\ell-1} \cdots s_n^{j_n}. 
    \end{equation}
    Evaluating its norm at radius $r$ yields 
    \begin{equation}
        \left\| \frac{\partial \chi_{k,i}}{\partial s_\ell} \right\|_r = \sum_{|j|=k, \, j_\ell \geq 1} |b_{j,i}| \cdot j_\ell \cdot r^{k-1}
    \end{equation}
    Summing over the row, we then obtain 
    \begin{equation}
        \sum_{\ell = 1}^n \left\| \frac{\partial \chi_{k,i}}{\partial s_\ell} \right\|_r = \sum_{\ell=1}^n \left( \sum_{|j|=k, \, j_\ell \geq 1} |b_{j,i}| \cdot j_\ell \cdot r^{k-1} \right)
    \end{equation}
    Since the sum is finite, we can simply factor
    \begin{equation}
        \sum_{\ell = 1}^n \left\| \frac{\partial \chi_{k,i}}{\partial s_\ell} \right\|_r = \left( \sum_{\ell=1}^n j_\ell \right) \left( \sum_{|j|=k} |b_{j,i}| \cdot r^{k-1} \right)
    \end{equation}
    Since $|j|=k$, we have
    \begin{equation}
        \sum_{\ell = 1}^n \left\| \frac{\partial \chi_{k,i}}{\partial s_\ell} \right\|_r = k \sum_{|j|=k} |b_{j,i}| r^{k-1} = \frac{k}{r} \sum_{|j|=k} |b_{j,i}| r^k = \frac{k}{r} \|\chi_k\|_r
    \end{equation}
    Taking the maximum over all components $i$, we obtain 
    \begin{equation}
        \|\dop\chi_k\|_r = \frac{k}{r} \|\chi_k\|_r 
    \end{equation}
    Since we need to evaluate this on a shrunken domain, let $r = \rho - \delta$. Then, we have
    \begin{equation}
        \|\dop\chi_k\|_{\rho-\delta} = \frac{k}{\rho-\delta} \|\chi_k\|_{\rho-\delta}
    \end{equation}
    Since all coefficients in the absolute norm are positive, shrinking the radius strictly decreases the norm, whence we have
    \begin{equation}
        \|\chi_k\|_{\rho-\delta} = \left( \frac{\rho-\delta}{\rho} \right)^k \|\chi_k\|_\rho 
    \end{equation}
    Now, substituting this into the previous expression, 
    \begin{equation}
        \|\dop\chi_k\|_{\rho-\delta} = \frac{k}{\rho-\delta} \left( \frac{\rho-\delta}{\rho} \right)^k \|\chi_k\|_\rho = \frac{k}{\rho} \left( \frac{\rho-\delta}{\rho} \right)^{k-1} \|\chi_k\|_\rho
    \end{equation}
    Since $(\rho-\delta)/\rho < 1$, $[(\rho-\delta)/\rho]^{k-1} < 1$, and we bound
    \begin{equation}
        \|\dop\chi_k\|_{\rho-\delta} \leq \frac{k}{\rho} \|\chi_k\|_\rho
    \end{equation}
    Substituting in the bound proved in (1), we obtain the result
    \begin{equation}
        \|\dop\chi_k\|_{\rho-\delta} \leq \frac{k^{\tau+1}}{\rho \cdot \gamma} \|R^{k,(k-1)}\|_\rho
    \end{equation}
    as desired. 

    We will prove (3) and (4) together. We first verify that $c_k(s)$ does not map points too far away, ensuring it stays within the original domain $\mathcal{D}_\rho$. Let $s \in \mathcal{D}_{\rho-\delta}$. By the triangle inequality, 
    \begin{equation}
        \|c_k(s)\| = \|s - \chi_k(s)\| \leq \|s\| + \|\chi_k(s)\|. 
    \end{equation}
    Since $s \in \mathcal{D}_{\rho-\delta}$, $\|s\| \leq \rho - \delta$. Since $\mathcal{D}_{\rho - \delta} \subset \mathcal{D}_\rho$, we can apply (2) and the added hypothesis to prove a useful function bound on $\|\chi_k\|_\rho$. By (2) and the added hypothesis, we have
    \begin{equation}
        \|\dop\chi_k\|_{\rho-\delta} \leq \frac{1}{\delta} \|\chi_k\|_\rho \leq \frac{1}{2} \implies \|\chi_k\|_\rho \leq \frac{\delta}{2},
    \end{equation}
    where we have set $\delta = \frac{\rho}{6N} < \frac{\rho}{k}$ (as we will have it throughout this section). Then, we have
    \begin{equation}
        \|c_k(s)\| \leq \|s\| + \|\chi_k(s)\| \leq (\rho - \delta) + \frac{\delta}{2} < \rho
    \end{equation}
    Thus, we have showed that $c_k(\mathcal{D}_{\rho-\delta}) \subset \mathcal{D}_\rho$. 

    Next, we will prove that $c_k$ is injective on the shrunken domain. To prove that $c_k$ is one-to-one on $\mathcal{D}_{\rho-\delta}$, we use the Jacobian bound. Let $s_1,s_2 \in \mathcal{D}_{\rho-\delta}$. Because the polydisc is convex, the straight line segment connecting $s_1$ and $s_2$ lies entirely in $\mathcal{D}_{\rho-\delta}$. So, we can apply the Mean Value Inequality for complex vector-valued functions:
    \begin{equation}
        \|\chi_k(s_1) - \chi_k(s_2)\| \leq \left( \sup_{s \in \mathcal{D}_{\rho-\delta}} \|\dop\chi_k(s)\| \right) \|s_1 - s_2\|
    \end{equation}
    By hypothesis, $\|\dop\chi_k(s)\| \leq 1/2$, hence
    \begin{equation}
        \|\chi_k(s_1) - \chi_k(s_2)\| \leq \frac{1}{2} \|s_1 - s_2\|
    \end{equation}
    Then, consider the distance between two transformed points:
    \begin{align}
        \|c_k(s_1) - c_k(s_2)\| &= \| (s_1-\chi_k(s_1)) - (s_2-\chi_k(s_2)) \| \\
        &\geq \|s_1 - s_2\| - \|\chi_k(s_1) - \chi_k(s_2)\| \\
        &\geq \|s_1 - s_2\| - \frac{1}{2} \|s_1 - s_2\| \\ 
        &= \frac{1}{2} \|s_1 - s_2\|
    \end{align}
    If $c_k(s_1) = c_k(s_2)$, then the left-hand side is zero, forcing $\|s_1 - s_2\| = 0$, implying $s_1 = s_2$. Therefore, $c_k$ is injective on $\mathcal{D}_{\rho-\delta}$. 

    To prove the existence of the inverse map $c_k^{-1}$, we will apply a contraction mapping argument. We need to guarantee that $c_k^{-1}$ is well-defined on a well-specified domain so that we can compose it with $F^{(k-1)}$. We will prove that the image of $c_k$ contains the polydisc $\mathcal{D}_{\rho-2\delta}$. Along these lines, let $w \in \mathcal{D}_{\rho-2\delta}$. We will show that there exists a unique $s \in \mathcal{D}_{\rho-\delta}$ such that $c_k(s) = w$. To do this, define the fixed point operator
    \begin{equation}
        s = w + \chi_k(s) \eqqcolon T_w(s)
    \end{equation}
    Apply the Banach Fixed Point Theorem to the operator $T_w(s)$ acting on the complete metric space $\mathcal{D}_{\rho-\delta}$:
    \begin{enumerate}
        \item[1.] $T_w$ maps $\mathcal{D}_{\rho-\delta}$ into itself. Let $s \in \mathcal{D}_{\rho-\delta}$, then
        \begin{equation}
            \|T_w(s)\| \leq \|w\| + \|\chi_k(s)\| \leq (\rho - 2\delta) + \frac{\delta}{2} < \rho - \delta.
        \end{equation}
        Hence, $T_w(\mathcal{D}_{\rho-\delta}) \subset \mathcal{D}_{\rho-\delta}$. 
        \item[2.] $T_w(s)$ is a strict contraction. For any $s_1,s_2 \in \mathcal{D}_{\rho-\delta}$, we have
        \begin{equation}
            \|T_w(s_1) - T_w(s_2)\| = \|\chi_k(s_1) - \chi_k(s_2)\| \leq \frac{1}{2} \|s_1 - s_2\|
        \end{equation}
    \end{enumerate}
    Because $T_w$ is a contraction mapping from $\mathcal{D}_{\rho-\delta}$ to itself, there exists a unique fixed point $s \in \mathcal{D}_{\rho-\delta}$ for every $w \in \mathcal{D}_{\rho-2\delta}$. Therefore, the inverse map $c_k^{-1}$ is rigorously well-defined as a biholomorphism from the smaller domain to the slightly larger domain:
    \begin{equation}
        c_k^{-1}: \mathcal{D}_{\rho-2\delta} \to \mathcal{D}_{\rho-\delta}
    \end{equation}
    completing the proof. 
\end{proof}

With the proof of Lemma 2 complete, we have rigorously quantified the geometric and analytic properties of the $k$-th order coordinate transformation. By sacrificing a small domain margin $\delta$, we guarantee that both the forward near-identity map $c_k$ and its inverse $c_k^{-1}$ are well-defined, strictly bounded, and analytically well-behaved. 

The next logical step in the normalization procedure is to actively apply this transformation to the mapping itself. In Lemma 3, we execute the full conjugation $G^{(k)} = c_k^{-1} \circ G^{(k-1)} \circ c_k$. This composition introduces spillover of high-order remainder terms. The following lemma tracks the accumulation of these polynomial deformations, yielding a strict upper bound on the coefficients of the updated map over the newly reduced domain. 

\begin{lemma}[Composition and Remainder Bound]
Assume the hypotheses of Lemma 2. Furthermore, assume the intermediate map $G^{(k-1)}$ satisfies the domain bound $\|G^{>2,(k-1)}\|_{\rho-2\delta} \leq \delta/2$. Then, the newly composed map after the $k$-th normalization step, $G^{(k)} = c_k^{-1} \circ G^{(k-1)} \circ c_k$ satisfies:
\begin{enumerate}
    \item[1.] \textbf{Domain of Analyticity (Topological Validity):} The full map $G^{(k)}$ is a well-defined analytic function on the shrunken polydisc $\mathcal{D}_{\rho-3\delta}$, and it maps this domain strictly into $\mathcal{D}_{\rho}$. 
    \item[2.] \textbf{Recursive Remainder Bound:} There exists a universal constant $C > 0$ such that the norm of the new remainder is bounded by:
    \begin{equation}
            \|R^{(k)}\|_{\rho_k} \leq \|R^{>k,(k-1)}\|_{\rho_{k-1}} + \frac{k^\tau}{\gamma} \|R^{k,(k-1)}\|_{\rho_{k-1}} + \|R^{k,(k-1)}\|_{\rho_{k-1}}^2 \left[ \frac{k^{\tau+1}}{\rho_{k-1} \gamma} + \frac{k^{2\tau}}{\gamma^2} \left( \frac{1}{\delta} + \frac{k}{\rho_{k-1}} \right) \right]
    \end{equation}
\end{enumerate}
\end{lemma}
\begin{proof}
To prove the first part, we must trace the map through the nested domains of:
\begin{equation}
    G^{(k)} = c_k^{-1}(G^{(k-1)}(c_k(s)))
\end{equation}
First, the forward transformation maps $s \mapsto c_k(s)$. For $s \in \mathcal{D}_{\rho - 3\delta}$, we have
\begin{equation}
    \|c_k(s)\|_{\rho - 3\delta} \leq \|s\|_{\rho-3\delta} + \|\chi_k(s)\|_{\rho-3\delta} \leq (\rho - 3\delta) + \delta = \rho - 2\delta
\end{equation}
Next, the old map application maps $x \mapsto G^{(k-1)}(x)$. For $x \in \mathcal{D}_{\rho-2\delta}$, we have
\begin{align}
    \|G^{(k-1)}(x)\|_{\rho-2\delta} &\leq \|\Lambda x\|_{\rho-2\delta} + \|G^{>2,(k-1)}(x)\|_{\rho-2\delta} \\
    &\leq \|x\|_{\rho-2\delta} + \|G^{>2,(k-1)}\|_{\rho-2\delta} \\
    &\leq (\rho-2\delta) + \frac{\delta}{2} \\
    &\leq \rho - \frac{3\delta}{2}
\end{align}
where the bound $\|G^{>2,(k-1)}\|_{\rho-2\delta} \leq \delta/2$ comes from the threshold condition. 
Then, by part 4 of Lemma 2, the inverse map is an analytical mapping with $c_k^{-1}(\mathcal{D}_{\rho-3\delta/2}) \subset \mathcal{D}_{\rho-\delta/2} \subset \mathcal{D}_\rho$, thus proving part 1. 

To prove the bound on the new remainder, we must carefully expand the conjugated map $G^{(k)} = c_k^{-1} \circ G^{(k-1)} \circ c_k$. We decompose the intermediate map $G^{(k-1)}$ into its linear part, the already-normalized (unavoidable resonant) terms up to degree $k-1$, and the remainder
\begin{equation}
    G^{(k-1)}(s) = \Lambda s + F^{(k-1)}(s) + R^{k,(k-1)}(s) + R^{>k,(k-1)}(s)
\end{equation}
We apply the forward transformation $c_k(s) = s - \chi_k(s)$ and the inverse transformation $c_k^{-1}(y) = y + \chi_k(y) + \eta_k(y)$. Substituting these into the composition yields
\begin{equation}
    G^{(k)}(s) = c_k^{-1}\left( \Lambda (s-\chi_k(s)) + F^{(k-1)}(s-\chi_k(s)) + R^{(k-1)}(s-\chi_k(s)) \right) 
\end{equation}
Applying the outer inverse map and gathering the terms by their polynomial degree, we can isolate the terms strictly of degree $k$. Note that because the lowest-degree term in $F^{(k-1)}$ is 3, its Taylor expansion $F^{(k-1)}(s-\chi_k(s)) = F^{(k-1)}(s) - \dop F^{(k-1)}(s)\chi_k(s) + \cdots$ only contributes terms of degree $k+1$ and higher. Hence, isolating the terms up to degree $k$, we obtain
\begin{equation}
    G^{(k)}(s) = \Lambda s + F^{(k-1)}(s) + \left[ \chi_k(\Lambda s) - \Lambda \chi_k(s) + R^{k,(k-1)}(s) \right] + R^{(k)}(s)
\end{equation}
By construction of the normal form, the generating function $\chi_k(s)$ is chosen specifically to solve the homological equation 
\begin{equation}
    \chi_k(\Lambda s) - \Lambda \chi_k(s) + R^{k,(k-1)}(s) = f_k(s)
\end{equation}
where $f_k(s)$ contains the unavoidable resonant terms of degree $k$. (Note that $f_k \equiv 0$ for all even $k$.) 

We then absorb $f_k(s)$ into the normalized part of the map, defining $F^{(k)}(s) = F^{(k-1)}(s) + f_k(s)$. The completely updated map is thus
\begin{equation}
    G^{(k)}(s) = \Lambda s  + F^{(k)}(s) + R^{(k)}(s)
\end{equation}
The new remainder $R^{(k)}(s)$ is composed strictly of the higher-order ``spillover'' terms. Grouping these terms, the updated remainder has the form
\begin{align}
    R^{(k)}(s) &= R^{>k,(k-1)}(s-\chi_k(s)) + \left[ R^{k,(k-1)}(s-\chi_k(s)) - R^{k,(k-1)}(s) \right] + \left[ \eta_k(y) - \chi_k(\Lambda s) \right] \nonumber \\
    &+ \left[ F^{>2,(k-1)}(s-\chi_k(s)) - F^{>2,(k-1)}(s) \right]  \\ 
    &\eqqcolon \text{I} + \text{II} + \text{III} + \text{IV}
\end{align}
Note that I is the shifted old remainder, II is the deformation of the $k$-th term, and III is the inverse map spillover, and IV is the normal form spillover. By the triangle inequality, 
\begin{equation}
    \|R^{(k)}\|_{\rho-3\delta} \leq \|\text{I}\|_{\rho-3\delta} + \|\text{II}\|_{\rho-3\delta} + \|\text{III}\|_{\rho-3\delta} + \|\text{IV}\|_{\rho-3\delta} 
\end{equation}
For the shifted old remainder, recall that $c_k(\mathcal{D}_{\rho-3\delta}) \subset \mathcal{D}_{\rho-2\delta}$. Hence, 
\begin{equation}
    \|R^{>k,(k-1)} \circ c_k\|_{\rho-3\delta} \leq \|R^{>k,(k-1)}\|_{\rho-2\delta} \leq \|R^{>k,(k-1)}\|_\rho
\end{equation}
For the deformation of the $k$-th term, we apply the Mean Value Theorem along a line segment connecting $s$ and $s-\chi_k(s)$. By I, both endpoints and the whole segment lie inside $\mathcal{D}_{\rho-2\delta}$. Hence, we have
\begin{equation}
    \|R^{k,(k-1)}(s-\chi_k(s)) - R^{k,(k-1)}(s)\|_{\rho-3\delta} \leq \left( \sup_{s \in \mathcal{D}_{\rho-2\delta}} \|\dop R^{k,(k-1)}\| \right) \|\chi_k(s)\|_{\rho-3\delta} 
\end{equation}
Applying the Cauchy estimate to the derivative
\begin{equation}
    \|\dop R^{k,(k-1)}\|_{\rho-2\delta} \leq \frac{k}{\rho} \|R^{k,(k-1)}\|_\rho 
\end{equation}
Note that the above Cauchy estimate is stronger than usual because the term is a finite polynomial having terms of only degree $k$. Then, substituting Lemma 2.1 to bound the generating function, 
\begin{equation}
    \|\chi_k\|_\rho \leq \frac{k^\tau}{\gamma} \|R^{k,(k-1)}\|_\rho,
\end{equation}
we obtain the bound on II as
\begin{equation}
    \|\text{II}\|_{\rho-3\delta} \leq \frac{k^{\tau+1}}{\rho \gamma} \|R^{k,(k-1)}\|_\rho^2
\end{equation}
Finally, for the spillover term, we add and subtract $\eta_k(\Lambda s)$ so that 
\begin{equation}
    \text{III} = \eta_k(y) - \chi_k(\Lambda s) = \left[ \eta_k(y) - \eta_k(\Lambda s) \right] + \left[ \eta_k(\Lambda s) - \chi_k(\Lambda s) \right] \eqqcolon \text{III}_a + \text{III}_b
\end{equation}
Let us consider first the III$_a$ term. Checking the domain, we need the distance between $\Lambda s$ and $y = \Lambda s - \Lambda \chi_k(s) + G^{>2,(k-1)}(s-\chi_k(s))$. Since $\Lambda$ is an isometry, and by our hypothesis on $\|G^{>2,(k-1)}\|_{\rho-2\delta} \leq \delta/2 < \delta$, we have
\begin{equation}
    \|y - \Lambda s\|_{\rho-3\delta} \leq \|\chi_k\|_{\rho-3\delta} + \|G^{>2,(k-1)}\|_{\rho-2\delta} \leq \delta + \delta/2
\end{equation}
Then, 
\begin{equation}
    \|y\| \leq \rho - 3\delta + \delta + \delta/2 = \rho-3\delta/2
\end{equation}
Hence, the line segment between $y$ and $\Lambda s$ stays inside $\mathcal{D}_{\rho-\delta}$. Since $\eta_k$ is an infinite series, we can't use the $k/\rho$ bound from Lemma 2; instead, we use the generic Cauchy estimate:
\begin{equation}
    \|\dop\eta_k\|_{\rho-3\delta} \leq \frac{1}{\delta} \|\eta_k\|_{\rho-2\delta} \leq \frac{1}{\delta} \|\chi_k\|_\rho,
\end{equation}
where the second inequality requires some proof. Recall that $\eta_k(y) = \chi_k(y+\eta_k(y))$. Assume that $y \in \mathcal{D}_{\rho-2\delta}$. By the threshold hypothesis, assume the deformations are small enough that $\|\eta_k\|_{\rho-2\delta} \leq \delta$. Then, $w = y + \eta_k(y) \in \mathcal{D}_{\rho-\delta}$, giving the bound
\begin{equation}
    \|\eta_k\|_{\rho-2\delta} = \|\chi_k(y + \eta_k(y))\|_{\rho-2\delta} \leq \|\chi_k\|_{\rho-\delta}
\end{equation}
Because $\chi_k$ is analytic on $\mathcal{D}_\rho$, its maximum on a smaller domain is bounded by its maximum on the larger domain. This proves the second inequality. Now, with the derivative bound and the distance bound, we can apply the Mean Value Theorem and the triangle inequality to obtain
\begin{align}
    \|\eta_k(y) - \eta_k(\Lambda s)\|_{\rho-3\delta} &\leq \frac{1}{\delta} \|\chi_k\|_\rho \left( \|\chi_k\|_\rho + \|G^{>2,(k-1)}\|_\rho \right) \\ 
    &\leq \frac{k^{2\tau}}{\delta \gamma^2} \|R^{k,(k-1)}\|_\rho^2 + \frac{k^\tau}{2 \gamma}  \|R^{k,(k-1)}\|_\rho 
\end{align}

Bounding III$_b$ is comparably routine:
\begin{align}
    \|\eta_k(\Lambda s) - \chi_k(\Lambda s)\|_{\rho-3\delta} &= \|\chi_k(\Lambda s + \eta_k(\Lambda s)) - \chi_k(\Lambda s)\|_{\rho-3\delta} \\
    &\leq \|\dop\chi_k\|_{\rho-\delta} \|\eta_k\|_{\rho-\delta} \\
    &\leq \|\dop\chi_k\|_{\rho-\delta} \|\chi_k\|_\rho \\
    &\leq \frac{k}{\rho} \|\chi_k\|_\rho^2 \\
    &\leq \frac{k^{2\tau+1}}{\rho \gamma^2} \|R^{k,(k-1)}\|_\rho^2 
\end{align}

The spillover generated by the nonlinear normal form terms is rather small, and bounded by
\begin{align}
    \|F^{>2,(k-1)}(s-\chi_k(s)) - F^{>2,(k-1)}(s)\|_{\rho-3\delta} &\leq \left( \sup_{s \in \mathcal{D}_{\rho-2\delta}} \| \dop F^{>2,(k-1)}\| \right) \|\chi_k(s)\|_{\rho-3\delta} \\ 
    &\leq \frac{B}{\delta} \|\chi_k(s)\|_{\rho-3\delta} \\ 
    &\leq \frac{1}{2}\|\chi_k(s)\|_{\rho-3\delta} \\ 
    &\leq \frac{k^\tau}{2\gamma} \|R^{k,(k-1)}\|_\rho
\end{align}
where $B$ is a constant bounding the nonlinear normal form terms on the larger domain, i.e., $\|F^{>2,(k-1)}\|_{\rho-2\delta} \leq B \leq \delta/2$, where the second inequality holds by hypothesis, the third comes from simplification, and the fourth comes from the bound on $\|\chi_k(s)\|_{\rho-2\delta}$ used previously. 

Putting all five bounds together, we obtain the desired result for part 2:
\begin{equation}
    \|R^{(k)}\|_{\rho-3\delta} \leq \|R^{>k,(k-1)}\|_\rho + \frac{k^\tau}{\gamma} \|R^{k,(k-1)}\|_\rho + \|R^{k,(k-1)}\|_\rho^2 \left[ \frac{k^{\tau+1}}{\rho \gamma} + \frac{k^{2\tau}}{\gamma^2} \left( \frac{1}{\delta} + \frac{k}{\rho} \right) \right]
\end{equation}
\end{proof}

With the completion of Lemma 3, we have successfully established the rigorous algebraic bounds for a single step of the normal form procedure. By tracking the deformations and spillovers, we have quantified exactly how much the remainder coefficients grow when transitioning from order $k-1$ to order $k$. However, to construct the final normal form up to an arbitrary order $N$, we must apply this transformation sequentially. The next critical step, addressed in Lemma 4, is to unroll this recursive relationship. By iterating the single-step bound from the initial base map up through $N$ steps, we derive a closed-form majorant that captures the cumulative, factorial growth of the coefficients resulting from the entire sequence of near-identity coordinate transformations. 

\begin{lemma}[Iterative Lemma]
Assume the hypotheses of Lemma 2 and 3 hold for all steps $k \leq N$. Define the quantities:
\begin{align}
    \rho_k &\coloneqq \rho_{k-1} - 3\delta, \qquad \delta = \frac{\rho}{6N}, \\
    \epsilon_k &= \|R^{(k)}\|_{\rho_k}
\end{align}

Define the linear small-divisor and quadratic penalties by
\begin{equation}
    L \coloneqq \frac{N^\tau}{\gamma}, \qquad Q \coloneqq \frac{2N^{\tau+1}}{\rho \gamma} + \frac{8 N^{2\tau+1}}{\rho \gamma^2},
\end{equation}
respectively. Suppose the initial perturbation $\epsilon_2$ satisfies the strict threshold condition
\begin{equation}
    \epsilon_2 \leq \frac{L}{Q(1+2L)^{N-2}}
\end{equation}
Then, for all $2 \leq k \leq N$, the quadratic remainder terms remain strictly bounded by the linear terms, $Q\epsilon_{k-1} \leq L$, and the sequence of remainders $\epsilon_k$ satisfies the recursive bound
\begin{equation}
    \epsilon_k \leq \epsilon_{k-1} (1+2L) \leq \epsilon_{k-1} \left( \frac{2N^\tau}{\gamma} \right)
\end{equation}
Consequently, unrolling this recursion yields the factorial-type bound:
\begin{equation}
    \epsilon_N \leq \epsilon_2 \left( \frac{2N^\tau}{\gamma} \right)^N
\end{equation}
\end{lemma}

\begin{proof}
The proof is split into a few steps. First, we must simplify and bound the coefficients from Lemma 3.2. Next, we establish the linear recursive inequality. Finally, we iterate $N-2$ times to obtain the desired result. 

Consider the $k$-th step. The domain is $\rho_k = \rho_{k-1} - 3\delta$. Applying Lemma 3 to the $k$-th step gives
\begin{equation}
    \|R^{(k)}\|_{\rho_k} \leq \|R^{>k,(k-1)}\|_{\rho_{k-1}} + \frac{k^{\tau}}{\gamma} \|R^{k,(k-1)}\|_{\rho_{k-1}} + \| R^{k,(k-1)} \|_{\rho_{k-1}}^2 \left[ \frac{k^{\tau+1}}{\rho_{k-1} \gamma} + \frac{k^{2\tau}}{\gamma^2} \left( \frac{1}{\delta} + \frac{k}{\rho_{k-1}} \right) \right]. 
\end{equation}
The first term is bounded above by $\epsilon_{k-1}$, since $R^{>k,(k-1)}$ is just the remainder $R^{(k-1)}$ without the degree $k$ terms. To bound the second term, we substitute $\|R^{k,(k-1)}\| \leq \epsilon_{k-1}$ to find
\begin{equation}
    \frac{k^\tau}{\gamma} \|R^{k,(k-1)}\|_{\rho_{k-1}} \leq \frac{N^\tau}{\gamma} \epsilon_{k-1}
\end{equation}
Finally, to bound the third term, we substitute $\delta = \rho/(6N)$, and use the bounds $1/\rho_{k-1} \leq 2\rho$ and $k \leq N$, to obtain
\begin{align}
    \| R^{k,(k-1)} \|_{\rho_{k-1}}^2 \left[ \frac{k^{\tau+1}}{\rho_{k-1} \gamma} + \frac{k^{2\tau}}{\gamma^2} \left( \frac{1}{\delta} + \frac{k}{\rho_{k-1}} \right) \right] &\leq \epsilon_{k-1}^2 \left[ \frac{2N^{\tau+1}}{\rho \gamma} + \frac{N^{2\tau}}{\gamma^2} \left( \frac{6N}{\rho} + \frac{2N}{\rho} \right) \right] \\ 
    &\leq \epsilon_{k-1}^2 \left[ \frac{2N^{\tau+1}}{\rho \gamma} + \frac{8N^{2\tau+1}}{\rho \gamma^2} \right]
\end{align}
Putting the three terms together, we obtain the simplified bound on the remainder:
\begin{equation}
    \epsilon_k \leq \epsilon_{k-1}\left( 1 + L + Q\epsilon_{k-1} \right),
\end{equation}
where $L$ and $Q$ are defined in the statement of the Lemma. To rigorously bound this by a purely growth factor, for example, 
\begin{equation}
    \epsilon_k \leq \epsilon_{k-1} (1+2L),
\end{equation}
we must strictly prove that $Q\epsilon_{k-1} \leq L$ at every step. We will prove this by induction. 

The base case is $k=2$. From our threshold hypothesis, we have
\begin{equation}
    \epsilon_2 \leq \frac{L}{Q(1+2L)^{N-2}}
\end{equation}
Since $1 + 2L > 1$ and $N \geq 2$, it is immediately true that $\epsilon_2 \leq L/Q$, i.e., $Q\epsilon_2 \leq L$. Substituting this into the remainder bound from Lemma 3 yields:
\begin{equation}
    \epsilon_3 \leq \epsilon_2 (1 + L + Q\epsilon_2) \leq \epsilon_2 (1+2L), 
\end{equation}
thus, showing the base case holds. 

Now, for the inductive step, assume that for all $j$ up to $k-1$, the sequence satisfies $\epsilon_j \leq \epsilon_{j-1}(1+2L)$. Unrolling this assumption from normalization steps 2 to $k-1$ gives the absolute bound:
\begin{equation}
    \epsilon_{k-1} \leq \epsilon_2 (1 + 2L)^{k-3}
\end{equation}
We must prove the bound holds for step $k$. First, we multiply both sides by $Q$: 
\begin{equation}
    Q\epsilon_{k-1} \leq Q \left[ \epsilon_2 (1 + 2L)^{k-3} \right]
\end{equation}
Since $k \leq N$, we have $(1+2L)^{k-3} \leq (1+2L)^{N-3} < (1+2L)^{N-2}$. Applying our initial hypothesis for $\epsilon_2$, we see that 
\begin{equation}
    Q\epsilon_{k-1} < Q \left( \frac{L}{Q(1+2L)^{N-2}} \right) (1+2L)^{N-2} = L,
\end{equation}
proving the inductive step. 

Since $Q\epsilon_{k-1} \leq L$ holds inductively, the recursion from Lemma 3 simplifies to
\begin{equation}
    \epsilon_k \leq \epsilon_{k-1} (1+L+Q\epsilon_{k-1}) \leq \epsilon_{k-1}(1+2L)
\end{equation}
Finally, since $\gamma < 1$, $N \geq 1$, $2L > 1$, we have that $1 + L \leq 2L$ and hence
\begin{equation}
    \epsilon_k \leq \epsilon_{k-1} (1+2L) \leq \epsilon_{k-1} (4L) = \epsilon_{k-1} \left( \frac{2N^\tau}{\gamma} \right)
\end{equation}
Unrolling the recursion, and using that $N-2 < N$ in the final bound, we obtain the desired result:
\begin{equation}
    \epsilon_N \leq \epsilon_2 \left( \frac{2N^\tau}{\gamma} \right)^N 
\end{equation} 
\end{proof}

\subsubsection{Truncation Remainder Bounds}

Lemma 4 provides a strictly algebraic bound on the size of the remainder coefficients after $N$ normalization steps. However, to evaluate the physical confinement of trajectories, we must translate this coefficient bound into a geometric bound on the actual remainder function evaluated over a specific physical domain. In Lemma 5, we transition from the space of formal coefficients back to the state space. By evaluating the remainder on a reduced polydisc of physical radius $a$, we establish a strict supremum bound on the magnitude of the perturbation that ultimately drives the trajectory away from surrounding sticky tori. This truncated remainder bound is the final component required before we can optimize the truncation order $N$ and deduce the effective stability time. 

\begin{lemma}[Truncated Remainder Bound]
Let the assumptions of Lemma 4 hold. Let $a > 0$ be a target physical radius such that $a < \rho/2$. Then, the remainder $R^{(N)}$ after $N$ normalization steps, evaluated on the strictly smaller polydisc $\mathcal{D}_a$, satisfies the exact algebraic bound:
\begin{equation}
    \|R^{(N)}\|_a \leq \epsilon_2 \left( \frac{2a}{\rho} \right) (AaN^\tau)^N, \qquad A = \frac{4}{\rho \gamma}
\end{equation}
\end{lemma}

\begin{proof}
By the geometric definition of domain shrinkage, the radius of analyticity after $N$ normalization steps is:
\begin{equation}
    \rho_N = \rho - \sum_{k=1}^N 3\delta = \rho - 3N \left( \frac{\rho}{6N} \right) = \rho - \frac{\rho}{2} = \frac{\rho}{2}
\end{equation}
Thus, by Lemma 4, the remainder $R^{(N)}$ is a well-defined analytic function on the polydisc $\mathcal{D}_{\rho/2}$ and its norm is bounded by 
\begin{equation}
    \|R^{(N)}\|_{\rho/2} = \epsilon_N \leq \epsilon_2 \left( \frac{2N^\tau}{\gamma} \right)^N
\end{equation}
Because $N$ steps of the normal form procedure have been completed, all non-resonant terms up to and including degree $N$ have been eliminated from the remainder. Consequently, the Taylor series expansion of $R^{(N)}$ around the origin begins with terms of at least degree $N+1$. 

By the general Schwarz Lemma (or standard Cauchy bounds), the norm of such a function on a smaller interior domain $\mathcal{D}_a$ (where $a < \rho_N$) scales strictly by the ratio of the radii to the power of the lowest non-vanishing degree. Therefore, projecting from $\rho_N$ down to $a$, we obtain:
\begin{equation}
    \|R^{(N)}\|_a \leq \|R^{(N)}\|_{\rho_N} \left( \frac{a}{\rho_N} \right)^{N+1}
\end{equation}
Substituting $\rho_N = \rho /2$ yields the spatial decay factor
\begin{equation}
    \|R^{(N)}\|_a \leq \epsilon_N \left( \frac{2a}{\rho} \right)^{N+1}
\end{equation}
Finally, inserting the recursive bound from Lemma 4, we obtain
\begin{equation}
    \|R^{(N)}\|_a \leq \epsilon_2 \left( \frac{2N^\tau}{\gamma} \right)^N \left( \frac{2a}{\rho} \right)^{N+1}
\end{equation}
Factoring out one power of $(2r\rho)$ aligns the exponents
\begin{equation}
    \|R^{(N)}\|_a \leq \epsilon_2  \left( \frac{2a}{\rho} \right) \left( \frac{4aN^\tau}{\rho \gamma} \right)^N
\end{equation}
Finally, defining the geometric constant $A = \frac{4}{\rho \gamma}$ completes the proof. 
\end{proof}

% III.B. Optimal Normalization Order 
\subsection{Optimal Normalization Order}

Having established a rigorous geometric bound on the truncated remainder in Lemma 5, we are now positioned to address the fundamental analytical competition inherent to asymptotic normal form series. The bound on the evaluated remainder $R^{(N)}$ over a physical polydisc of radius $a$ exhibits two strictly opposing behaviors as the normalization order $N$ increases. On one hand, advancing the normal form to a higher degree pushes the un-normalized perturbation to a higher polynomial order. Because the domain radius $a$ is chosen to be small, this causes the magnitude of the remainder to decay geometrically toward zero. On the other hand, the iterative accumulation of small divisors from the homological equations guarantees that the majorant coefficients bounding this remainder grow factorially with $N$. Consequently, the infinite series is formally divergent \cite{giorgilli2022notes}. For any fixed, non-zero physical radius $a$, the factorial growth of the coefficients will inevitably overtake the geometric decay of the high-order polynomials. To obtain the strongest possible finite-time stability bound, we must halt the normalization procedure at the precise threshold where this divergence begins. This optimal truncation order, $N_{\text{opt}}$, minimizes the magnitude of the remainder and is intrinsically coupled to the chosen domain radius. We formalize this minimization in the following proposition. 

\begin{proposition}[Optimal Truncation]
Assume the remainder after $N$ normalization steps satisfies the bound:
\begin{equation}
    \|R^{(N)}\|_a \leq \epsilon_2 \left( \frac{2a}{\rho} \right) (AaN^\tau)^N, \qquad A = \frac{4}{\rho \gamma}
\end{equation}
with $\epsilon_2, \rho, \gamma, \tau > 0$ constants independent of $N$ and $a$. For a fixed physical radius $a$, satisfying the strict algebraic condition
\begin{equation}
    a \leq \frac{1}{A(2e)^\tau}
\end{equation}
the optimal number of normalization steps, $N_{\text{opt}} \geq 2$, that minimizes the remainder bound is
\begin{equation}
    N_{\text{opt}} = \left\lfloor \frac{1}{e} \left( \frac{1}{Aa} \right)^{\frac{1}{\tau}} \right\rfloor 
\end{equation}
Evaluated at $N_{\text{opt}}$, the remainder is exponentially small with respect to $a$, satisfying
\begin{equation}
    \|R^{(N)}\|_a \leq \epsilon_2 \left( \frac{2a}{\rho} \right) \exp(-\tau N_{\text{opt}})
\end{equation}
\end{proposition}

\begin{proof}
To prove the upper bound of the remainder, we isolate the factor that depends on the normalization order $N$. Let
\begin{equation}
    g(N) = (AaN^\tau)^N,
\end{equation}
where $N$ here is thought of as a continuous (i.e., not discrete) variable. Since the pre-factor $\epsilon_2 (2a/\rho) > 0$ and is independent of $N$, it suffices to minimize $g$. Moreover, since $g(N) > 0$ for $N \geq 1$, the minimum of $g(N)$ coincides with the minimum of $\ln g(N)$, so we will minimize the latter:
\begin{equation}
    \ln g(N) = N \ln (AaN^\tau) = N \ln (Aa) + \tau N \ln N
\end{equation}
Computing the derivative with respect to $N$, we get 
\begin{equation}
    \frac{d}{dN} \ln g(N) = \ln (AaN^\tau) + \tau 
\end{equation}
Setting the derivative to zero yields
\begin{equation}
    \ln(AaN^\tau) = -\tau, 
\end{equation}
and exponentiating gives 
\begin{equation}
    AaN^\tau = e^{-\tau}
\end{equation}
Solving for $N$, we obtain
\begin{equation}
    N = \frac{1}{e} \left( \frac{1}{Aa} \right)^{\frac{1}{\tau}}
\end{equation}
Note that this is a minimum by 
\begin{equation}
    \frac{d^2}{dN^2} \ln g(N) = \frac{\tau}{N} > 0
\end{equation}
Because the normal form procedure requires a natural number of steps, we select the optimal discrete order by taking the floor
\begin{equation}
    N_{\text{opt}} = \left\lfloor \frac{1}{e} \left( \frac{1}{Aa} \right)^{\frac{1}{\tau}} \right\rfloor 
\end{equation}
Let us show now that $N_{\text{opt}} \geq 2$. If $a \leq \frac{1}{A(2e)^\tau}$, we can simply plug in this value of $a$ to the expression for $N_{\text{opt}}$ to find 
\begin{equation}
N_{\text{opt}} \geq \left\lfloor \frac{1}{e} \left( \frac{1}{A\left(\frac{1}{A(2e)^\tau}\right)} \right)^{\frac{1}{\tau}} \right\rfloor = \left\lfloor \frac{1}{e} (2e)^{\tau \cdot \frac{1}{\tau}} \right\rfloor = 2
\end{equation}

To compute the magnitude of the remainder at this optimal order, we substitute the exact optimal relationship back into the continuous function $g(N)$:
\begin{equation}
    g(N_{\text{opt}}) = (AaN_{\text{opt}})^{N_{\text{opt}}} \leq (e^{-\tau})^{N_{\text{opt}}} = \exp(-\tau N_{\text{opt}})
\end{equation}
Multiplying by the $N$-independent pre-factor, the total optimal remainder bound is
\begin{equation}
    \|R^{(N)}\|_a \leq \epsilon_2 \left( \frac{2a}{\rho} \right) \exp(-\tau N_{\text{opt}})
\end{equation}
This completes the proof. 
\end{proof}

With the derivation of the optimal truncation order and the corresponding exponentially small upper bound on the remainder, our strictly algebraic quantification of the normal form procedure is complete. Throughout this section, we have systematically traced the mapping from its local Taylor expansion through a sequence of near-identity symplectic conjugations, rigorously accounting for both the complex deformation of the analytical domain and the factorial accumulation of small divisors. By stopping the normalization process at $N_{\text{opt}}$, we have successfully constructed a finite-order, integrable approximation of the local dynamics whose truncation error is bounded by a quantity that is exponentially small with respect to the distance from the elliptic fixed point. 

% IV. Radial Drift and Effective Stability Estimates 
\section{Radial Drift and Effective Stability Estimates}
In the preceding sections, we have systematically reduced the local dynamics around the elliptic fixed point into a finite-order normal form. By optimizing the truncation order $N_{\text{opt}}$, we established that the truncated remainder, i.e., the un-normalized perturbation, is bounded by a quantity that is exponentially small with respect to the radius of the chosen physical domain. Up to this point, our analysis has been strictly algebraic, focusing on the convergence properties and coefficient growth of the formal power series. 

We now transition to physical stability. The central objective of this section is to translate the exponentially small remainder bound into rigorous confinement guarantees for trajectories near elliptic fixed points of (symplectic) Poincar\'e maps. In a perfectly normalized system governed solely by the integrable twist map, the orbital radii (amplitudes) are exact integrals of motion, and trajectories are perpetually confined to invariant tori. In the full nonlinear system, however, the truncated remainder acts as a persistent perturbation that allows the orbital radii to slowly drift over time, potentially leading to instability. 

By quantifying the maximum rate of this radial drift, we can bound the total deviation of a trajectory over macroscopic timescales. This approach culminates in a Nekhoroshev-type effective stability theorem, yielding an exponentially long lower bound on the time a trajectory is guaranteed to remain confined within a specified physical neighborhood. 

\subsection{Bounds on the Radial Drift}

To derive the macroscopic stability time, we must first establish the worst-case rate at which a trajectory can deviate from its unperturbed invariant torus. We define this deviation strictly in terms of the orbital radii $r \in \mathbb{R}^n$, where $r_j = |z_j|$. Under the iterated map, $G^{(N_{\text{opt}})}$, the normal form terms $F$ perfectly preserve the distance from the origin, contributing only pure phase rotations. Therefore, any change in the amplitude (i.e., the radial drift $\Delta r$) is driven exclusively by the remainder $R^{(N_{\text{opt}})}$. By projecting our optimally bounded remainder onto the radial coordinates, we can formulate a strict geometric upper bound on the change in the orbital radius per iteration of the return map. 

\begin{proposition}[Radial Drift Bound]
Let $(z,\bar{z})$ be the complexified normal form coordinates in the local polydisc $\mathcal{D}_r$, where $z_i = x_i + \imag y_i$. For each degree of freedom, $j = 1,\ldots,n$, we define the orbital radius in the $i$-th phase plane as $r_i^2 = z_i \bar{z}_i$. Under the discrete Poincar\'e map $G^{N_{\text{opt}}}$, the one-step drift in the radius $\Delta r_i \coloneqq r_i^{(1)} - r_i^{(0)}$ is strictly bounded by the exponentially small remainder. Specifically, for any point in $\mathcal{D}_a$:
\begin{equation}
|\Delta r_i| \leq \| R^{(N_{\text{opt}})} \|_a
\end{equation}
where $\| R^{(N_{\text{opt}})} \|_r \leq \epsilon_{N_{\text{opt}}}$, the exponentially small upper bound derived in Proposition 5. 
\end{proposition}
\begin{proof}
In normalized coordinates, the Poincar\'e map decomposes into an integrable twist map, $F^{(N_{\text{opt}})}$, and the truncated remainder, $R^{(N_{\text{opt}})}$, where the $i$-th element of the integrable twist map is given by:
\begin{equation}
    F_i^{(N_{\text{opt}})}(z^{(0)},\bar{z}^{(0)}) = z_i^{(0)} \exp\left[\imag \omega_i(r) \right] 
\end{equation}
Consequently, the twist map is an exact isometry on the radii, hence preserving the radius
\begin{equation}
    |F_i^{(N_{\text{opt}})}| = |z_i^{(0)}| = r_i^{(0)}
\end{equation}
Applying the map, the updated $i$-th coordinate is
\begin{equation}
    z_i^{(1)} = F_i^{(N_{\text{opt}})}(z^{(0)},\bar{z}^{(0)}) + R_i^{(N_{\text{opt}})}(z^{(0)},\bar{z}^{(0)})
\end{equation}
To bound the drift of the radius, $r_i^{(1)} = |z_i^{(1)}|$, we just apply the reverse triangle inequality
\begin{equation}
|r_i^{(1)} - r_i^{(0)}| = \left| |z_i^{(1)}| - |F_i^{(N_{\text{opt}})}| \right| \leq \left| z_i^{(1)} - F_i^{(N_{\text{opt}})} \right| = |R_i^{(N_{\text{opt}})}|
\end{equation}
Taking the supremum over the domain $\mathcal{D}_a$ immediately yields the result
\begin{equation}
    |\Delta r_i| \leq \| R_i^{(N_{\text{opt}})}\|_a
\end{equation}
as desired. 
\end{proof}

Proposition 7 rigorously establishes that the single-step deviation in the orbital radii is constrained by the same exponentially small bound that governs the optimal remainder. Because the normal form perfectly decouples the amplitude from the phase up to order $N_{\text{opt}}$, the primary mechanism for instability is reduced to an exceptionally slow diffusion driven by this perturbation. 

% IV.B. Stability Time and Confinement Region 
\subsection{Stability Time and Confinement Region}
With the maximum single-step radial drift quantified, the final requirement is to project this discrete drift over a long sequence of iterations. By determining how many iterations it takes for this exponentially small radial drift to accumulate to a macroscopic threshold--specifically, the point at which a trajectory crosses the boundary of our defined analytical domain. By dividing the total allowable spatial deviation by the maximum single-step drift, we arrive at the central dynamical result of this local analysis: a Nekhoroshev-type stability estimate. 

\begin{theorem}[Effective Stability of an Elliptic Fixed Point of a Poincar\'e Map]
Let $G: \mathcal{D}_\rho \to \mathbb{C}^{2n}$ be a symplectic, real-analytic Poincar\'e map with an elliptic fixed point at the origin. Assume the linearized frequencies $\omega \in \mathbb{R}^n$ satisfy a Diophantine condition with constants $\gamma > 0$ and $\tau \geq n$. Let the nonlinear part of the map be bounded by $M$ on $\mathcal{D}_\rho$. 

Let $a$ be a target confinement radius satisfying the strict threshold condition $a \leq a_0$, where $a_0 = \frac{1}{A(2e)^\tau}$ guarantees both the convergence of the normal form construction and a strictly positive optimal normalization order $N_{\text{opt}} \geq 2$. 

Then, there exists a near-identity polynomial change of coordinates to normal form such that for any initial state $z^{(0)} \in \mathcal{D}_{a/2}$, the orbital radii $r_j^{(k)} = |z_j^{(k)}|$ under the $k$-th iteration of the map $(z^{(k)},\bar{z}^{(k)}) = G^k(z^{(0)},\bar{z}^{(0)})$ remain tightly bounded. Specifically, the trajectory remains confined to the domain $\mathcal{D}_a$, satisfying
\begin{equation}
    |r_i^{(k)} - r_i^{(0)}| < \frac{a}{2} 
\end{equation}
for all degrees of freedom $i = 1,\ldots,n$, provided the number of map iterations $k$ does not exceed the effective stability time, i.e., $k \leq T_{\text{eff}}(a)$, which is explicitly given by
\begin{equation}
T_{\text{eff}}(a) = C \exp \left( \tau \left\lfloor \frac{1}{e} \left( \frac{1}{Aa} \right)^{\frac{1}{\tau}} \right\rfloor \right)
\end{equation}
where, $A = \frac{4}{\rho \gamma}$ and $C = \frac{\rho}{4\epsilon_2}$ are strictly positive constants. 
\end{theorem}
\begin{proof}
The structure of the proof is as follows. First, we will argue on the validity of the optimal normalization. Next, we start from single-step drift, implement the optimal remainder bound, and accumulate the drift over $k$ iterations of the map to compute the $k$-th step drift accumulation. Finally, we prove the effective stability time. 

Let the target confinement radius $a$ be fixed such that $a \leq a_0 = \frac{1}{A(2e)^{\tau}}$. By Proposition 6, this algebraic condition guarantees that the continuous minimum of the remainder occurs at $N_{\text{opt}} \geq 2$. Further, $a \leq a_0$ ensures that the initial macroscopic perturbation $\epsilon_2 = \|R^{(2)}\|_\rho$ is sufficiently small to satisfy the strict local inductive threshold condition required by Lemma 4. So, by Lemmas 2 and 3, the sequence of near-identity coordinate transformations $c_k(z,\bar{z})$ for $k = 2,\ldots,N_{\text{opt}}$ is topologically valid and analytically bounded on the nested sequence of shrinking polydiscs. 

Following the $N_{\text{opt}}$-th normalization step, we evaluate the remainder of the Poincar\'e map on the physical target domain $\mathcal{D}_a$. By Lemma 5, the spatial projection from the macroscopic domain $\rho/2$ down to $a$ introduces a geometric decay factor. Applying the optimization from Proposition 6, the norm of the remainder is strictly bounded by an exponentially small quantity:
\begin{equation}
\|R^{(N_{\text{opt}})}\|_a \leq \epsilon_2 \left( \frac{2a}{\rho} \right) \exp(-\tau N_{\text{opt}}) \eqqcolon \epsilon_{\text{opt}}
\end{equation}

Let $z^{(0)} \in \mathcal{D}_{a/2}$ be the initial condition in the normal form coordinates, which physically requires that the initial orbital radii satisfy
\begin{equation}
    r_i^{(0)} \leq a/2 \qquad \forall i=1,\ldots,n
\end{equation}
By Proposition 7, the one-step drift in the radius under the discrete normalized map $G^{(N_{\text{opt}})}$ is governed exclusively by the remainder, as the integrable twist map perfectly preserves the moduli of the coordinates. Thur, for any state in $\mathcal{D}_a$, the drift over a single crossing of the Poincar\'e map is bounded by
\begin{equation}
    \left| r_i^{(m+1)} - r_i^{(m)} \right| \leq \epsilon_{\text{opt}}
\end{equation}
Assume inductively that the trajectory $(z^{(m)},\bar{z}^{(m)})$ remains confined within $\mathcal{D}_a$ for all steps $m < k$. The total accumulated drift after $k$ iterations is then bounded by, applying the triangle inequality, 
\begin{equation}
    \left| r_i^{(k)} - r_i^{(0)} \right| \leq \sum_{m=0}^{k-1} \left| r_i^{(m+1)} - r_i^{(m)} \right| \leq k \epsilon_{\text{opt}}
\end{equation}

For the trajectory to remain confined strictly to $\mathcal{D}_a$, the total accumulation must be strictly less than the minimal distance from the initial domain $\mathcal{D}_{a/2}$ to the boundary of $\mathcal{D}_a$, i.e., $a-a/2 = a/2$. Therefore, the topological confinement is rigorously guaranteed for any number of iterations $k$ satisfying
\begin{equation}
    k\epsilon_{\text{opt}} < \frac{a}{2} \implies k < \frac{a}{2\epsilon_{\text{opt}}}
\end{equation}
Hence, the effective stability time is defined by
\begin{align}
    T_{\text{eff}}(a) &\coloneqq \frac{a}{2\epsilon_{\text{opt}}} \\
    &= \frac{a}{2\left[ \epsilon_2 \left( \frac{2a}{\rho} \right) \exp(-\tau N_{\text{opt}}) \right]} \\
    &= \frac{\rho}{4\epsilon_2} \exp(\tau N_{\text{opt}}) \\
    &= C \exp \left( \tau \left\lfloor \frac{1}{e} \left( \frac{1}{Aa} \right)^{\frac{1}{\tau}} \right\rfloor \right),
\end{align}
as was to be shown. 
\end{proof}

Theorem 8 provides a powerful temporal guarantee, establishing the instability in the immediate vicinity of the elliptic fixed point is an exceptionally slow process. However, for practical applications in astrodynamics, it is equally important to understand the spatial implications of this temporal bound. If a trajectory is guaranteed to remain stable for an exponentially long time, what is the precise region that contains it during that time? We address this, following in the tradition of many authors, in defining the effective stability region. By restricting the initial condition to a slightly smaller concentric subdomain, the slow radial drift guaranteed by Theorem 8 ensures that the trajectory will never reach the outer boundary of the valid normal form domain before the effective stability time. 

Corollary 9 formally translates the temporal bound of Theorem 8 into this spatial confinement guarantee, defining the strict physical envelope within which the local dynamics remain predictably bounded. 

\begin{corollary}[Confinement Region]
Let $T > C$ be a prescribed minimum effective stability time, representing the desired number of iterations of the Poincar\'e map. To strictly guarantee confinement of the discrete trajectory for all $k \leq T$, the target confinement radius $a$ must be bounded from above by
\begin{equation}
    a(T) \leq \frac{1}{Ae^{\tau}\left(1 + \frac{1}{\tau} \ln \left(\frac{T}{C}\right)\right)^\tau}
\end{equation}
where $A = \frac{4}{\rho \gamma}$ and $C = \frac{\rho}{4\epsilon_2}$. 
\end{corollary}
\begin{proof}
To guarantee confinement for at least $T$ iterations, we require $T_{\text{eff}}(a) \geq T$. For a continuous bound (that is simpler to use in the example of the following section), we use $\lfloor x \rfloor > x - 1$. By this inequality and Theorem 8, we have
\begin{equation}
    T_{\text{eff}}(a) > C\exp(\tau(x-1)) = Ce^{-\tau}\exp(\tau x), \qquad x = \frac{1}{e} \left( \frac{1}{Aa} \right)^{\frac{1}{\tau}}
\end{equation}
Enforcing the stricter condition leads to:
\begin{equation}
    Ce^{-\tau} \exp(\tau x) \geq T
\end{equation}
Rearranging and solving for $a$ yields
\begin{equation}
    a \leq \frac{1}{Ae^{\tau}\left(1 + \frac{1}{\tau} \ln \left(\frac{T}{C}\right)\right)^\tau}
\end{equation}
as desired. 
\end{proof}

With the establishment of the effective stability time and its associated spatial confinement region, the theoretical framework of our local analysis is complete. The results of this section guarantee that trajectories initialized sufficiently close to an elliptic fixed point of a Poincar\'e map will not experience sudden chaotic diffusion, but rather are strictly bounded by an exponentially slow radial drift. However, while Theorem 8 and Corollary 9 provide rigorous analytical guarantees, these bounds are expressed in terms of the abstract geometric properties of the normal form and the optimal truncation order $N_{\text{opt}}$. To understand the true physical scale of these confinement regions for Near-Rectilinear Halo Orbits in the Earth-Moon system, we must evaluate these series numerically. In the following section, we transition from theoretical estimates to computational implementation, applying this mathematical framework directly to the elliptic sub-family of NRHOs to extract quantitative, physically meaningful effective stability regions. 

%%%%%%%%%%%%%%%%%%%%%%%%%%%%%%%%%%%%%%%%%%%%%%%%%%%%%%%%%%%%%%%%%
% V. Application to the Earth-Moon Near-Rectilinear Halo Orbits %
%%%%%%%%%%%%%%%%%%%%%%%%%%%%%%%%%%%%%%%%%%%%%%%%%%%%%%%%%%%%%%%%%
\section{Application to the Earth-Moon $L_2$ Near-Rectilinear Halo Orbits}
Having established the rigorous analytical framework for the effective stability of discrete Poincar\'e maps via the normal form construction detailed in the previous sections, we now apply these bounds to a concrete astrodynamical scenario. In this section, we investigate the effective stability of the linearly stable Earth-Moon $L_2$ Near-Rectilinear Halo Orbits (NRHOs) in the Circular Restricted 3-Body Problem (CR3BP). 

Since the NRHOs exist within a highly sensitive dynamical regime characterized by adjacent regions of strong instability and close lunar passages, bounding the long-term behavior of trajectories in their vicinity is a delicate nonlinear problem. To transition our theoretical estimates into computable physical bounds, we utilize jet transport to expand the Poincar\'e return map of a chosen NRHO up to an arbitrary truncation order $N$. For the spatial CR3BP, the reduction to a Poincar\'e section yields a 4-dimensional symplectic map, allowing us to set the Diophantine exponent to $\tau = 2$. 

By extracting the exact Taylor coefficients via jet transport, we can systematically evaluate the analytical constants required by the iterative lemma. Specifically, we compute the domain of analyticity $\rho$ via the (Cauchy-Hadamard) root test, empirically bound the smallest non-resonant divisor up to order $N$, and saturate the remainder bound $\epsilon_2$. This permits the direct computation of the strict confinement radius $a(T)$ as a function of discrete orbital revolutions $T$, mathematically guaranteeing a stability envelope around the chosen NRHO. 

The primary objective of this numerical application is to demonstrate that the majorant series derived herein do not merely provide asymptotic existence proofs, but rather yield physically meaningful, macroscopic stability volumes over baseline mission durations. We target a nominal mission lifetime of $T_{\text{eff}} = 15$ years to show that a well-chosen truncation order guarantees strict confinement within practical operational tolerances. 

\subsection{Dynamical Model and Reference Orbit}
To evaluate the effective stability bounds numerically, we consider the spatial Circular Restricted 3-Body Problem (CR3BP) applied to the Earth-Moon system. The primary bodies, Earth and Moon, are assumed to move in circular orbits about their common barycenter. We use the standard synodic rotating frame, normalizing the distance between the primaries, their total mass, and the system's mean motion to unity. The motion of the satellite is governed by the equations of motion
\begin{align}
    \ddot{x} - 2\dot{y} &= \frac{\partial U}{\partial x}, \nonumber \\
    \ddot{y} + 2\dot{x} &= \frac{\partial U}{\partial y}, \\
    \ddot{z} &= \frac{\partial U}{\partial z}, \nonumber
\end{align}
where the effective potential is given by $U(x,y,z) = \frac{1}{2}(x^2 + y^2) + \frac{1-\mu}{r_1} + \frac{\mu}{r_2}$, where we place the Earth at $x = -\mu$ and the Moon at $x = 1-\mu$ so that
\begin{equation}
    r_1 = \sqrt{(x+\mu)^2 + y^2 + z^2}, \quad r_2 = \sqrt{(x-1+\mu)^2 + y^2 + z^2}
\end{equation}
and where we use the Earth-Moon mass ratio provided by the Jet Propulsion Laboratory (NASA/Caltech):
\begin{equation}
    \muem = 0.01215058560962404.
\end{equation}

To cast the continuous flow as a discrete symplectic map suitable for normal form analysis, we construct a Poincar\'e surface of section $\Sigma$. For the $L_2$ NRHOs, it is convenient to define this section at apoapsis, i.e., $\{ y = 0, \, \dot{y} < 0 \}$, to capture the state once per orbital revolution. Additionally, fixing the Jacobi constant at $C_0$ of the periodic orbit reduces the map to a 4-dimensional symplectic map. The continuous trajectories of the CR3BP are integrated between successive crossing of $\Sigma$ using jet transport, which generates a high-order Taylor expansion of the return map $G(s)$ mapping an initial variation $s \in \Sigma$ to its subsequent return. 

For this numerical application, we select a specific, linearly stable (i.e., normally elliptic) NRHO belonging to the $L_2$ halo orbit family, shown in Figure \ref{fig:ExampleFigure}; without loss of generality, we take from the Southern $L_2$ halo orbit family, representative of orbits near the 9:2 synodic resonant orbit slated as Gateway's nominal trajectory. The fixed point of this periodic orbit on the surface of section, the period, Jacobi constant $C$, and the Floquet exponents $\lambda_{1,2} = e^{\pm \imag\alpha_1}, \lambda_{3,4} = e^{\pm \imag\alpha_2}$ are presented in Table \ref{tab:OrbitData}. Note that the fixed point data was originally taken from the JPL periodic orbit database \cite{ssd_periodic_orbits_2026}, then corrected to higher precision. Further, note that the period of the orbit is about 10 days (9.99\ldots). 

\begin{table}[ht]
    \centering
    \[\begin{array}{|c|l|}
    \hline
       x _0  & \texttt{ 1.07442699836894799731734842966765573872} \\
       y _0  & \texttt{ 0} \\
       z _0 &\texttt{-2.02077758410419468689345932453312581546e-1} \\
       \dot x _0 & \texttt{ 5.38129566215790323615561941446878151138e-37} \\
       \dot y _0 & \texttt{-1.91460361329368351175536906288136125112e-1} \\
       \dot z _0 & \texttt{ 1.05394660456707887877017392859740255247e-36} \\ \hline
       T & \texttt{ 2.25450199926392082895876613944056856405} \\
       C & \texttt{ 3.0159} \\\hline
       \alpha _1 & \texttt{ 2.335685763037916} \\\hline
       \alpha _2 & \texttt{ 1.571477893259234} \\\hline
    \end{array}\]
    \caption{Elliptic fixed point using 124 bits (37 digits), newton tolerance of $10^{-35}$, fix Jacobi constant $C$, and $\mu = \muem$; and $\alpha _{1,2}$ are the Floquet exponents in $[-\pi,\pi)$. Note that $\alpha_2 = \frac{\pi}{2} + \mathcal{O}(10^{-4})$.} \label{tab:OrbitData}
\end{table}

\begin{figure}[ht]
    \centering
    \begin{subfigure}{.45\textwidth}
        \centering
        \includegraphics[width=\linewidth]{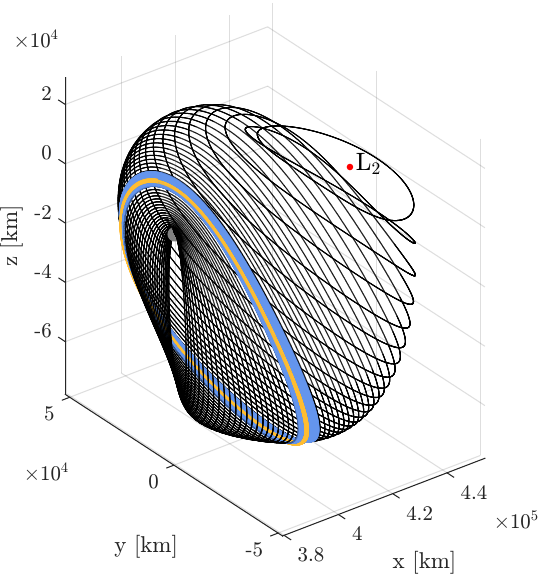}
        \caption{EM $L_2$ Southern Halo Orbits}
        \label{fig:OrbitPlot_ChosenOrbit}
    \end{subfigure}%
    \hfill
    \begin{subfigure}{.53\textwidth}
        \centering
        \includegraphics[width=\linewidth]{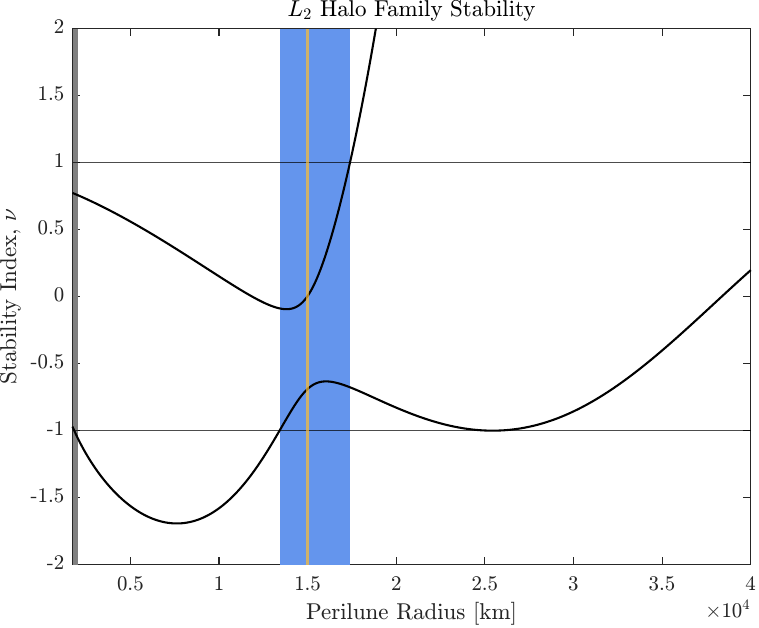}
        \caption{Stability of $L_2$ Halo Orbits}
        \label{fig:StabilityPlot_ChosenOrbit}
    \end{subfigure}
    \caption{Selected $L_2$ Southern NRHO for effective stability computation (orange), surrounded by blue region of normally elliptic NRHOs.}
    \label{fig:ExampleFigure}
\end{figure}

\subsection{Effective Stability Computation}
To demonstrate the practical application of our theoretical framework, we evaluate the effective stability bounds for a selected Earth-Moon $L_2$ Southern NRHO with a period of approximately 10 days. The objective is to compute the confinement region $a(T)$ that guarantees a stability boundary for a desired mission lifetime $T$, translating the Nekhoroshev-like estimates into physical spatial coordinates. 

The computation of the effective stability region relies on evaluating the sequence of intermediate variables defined by our rigorous bounds in Section 4. First, the Taylor expansion of the unperturbed Poincaré map $G(s)$ is computed to a high order. From these numerical coefficients, we extract the necessary constants the optimal remainder bound: 
\begin{itemize}
    \item \textbf{Analyticity Domain:} The radius of the polydisc on which the map is analytic is estimated using the Cauchy-Hadamard root test, $\rho \approx \|G^N\|^{-1/N}$. We choose $N$ sufficiently large such that $\rho$ flattens out, as the approximation can suffer at low $N$. 
    \item \textbf{Cumulative Small Divisor:} The uniform Diophantine constant required for the bound at order $N$ is computed as the minimum of all small divisors encountered up to that step:
    \begin{equation}\label{eqn:SmallDivisor}
        \gamma = \min_{\substack{2 \le |j| \le N \\ i \in \{1, \dots, 2n\} \\ \text{non-resonant}}} \Big( |\lambda_i - \lambda^j| \cdot |j|^\tau \Big)
    \end{equation}
    \item \textbf{Iterative Lemma Constant:} The constant $\epsilon_2$ is computed by the formula of the iterative lemma:
    \begin{equation}
        \epsilon_2 \leq \frac{L}{Q(1+2L)^{N-2}}, \quad L \coloneqq \frac{N^\tau}{\gamma}, \quad Q \coloneqq \frac{2N^{\tau+1}}{\rho \gamma} + \frac{8 N^{2\tau+1}}{\rho \gamma^2},
    \end{equation}
    with $N$ the truncation order, and $\tau$, $\gamma$, and $\rho$ previously computed. We will use that $\epsilon_2$ is equal to the right-hand side, not $\leq$ to compute the largest possible guaranteed stability region. 
\end{itemize}
Putting together the above intermediate variables, we compute the final constants used in the explicit bound:
\begin{itemize}
    \item \textbf{Majorant Coefficient:} The constant governing the growth of the remainder at order $N$ is given by:
    \begin{equation}
        A = \frac{4}{\rho \gamma}
    \end{equation}
    \item \textbf{Characteristic Nekhoroshev Accumulation Time:} Finally, the constant $C$ is compute as:
    \begin{equation}
        C = \frac{\rho}{4\epsilon_2}
    \end{equation}
\end{itemize}
To extract the strongest possible finite-time stability guarantee, we must halt the normalization procedure at the optimal truncation order $N_{\text{opt}}$. This is defined as the order that minimizes the coefficient $A_k$, thereby minimizing the overall bound on the truncated remainder. As illustrated in Figure \ref{fig:SmallDivisor}, the small divisor $\gamma_k$ takes its minimum at $N=3$ due to a low-order resonance and remains flat for subsequent orders. Simultaneously, shown in Figure \ref{fig:RoC_G}, the sequence $\rho_k$ generally decreases before asymptoting toward the true fixed radius of convergence $\rho$. Note that the minimum occurs at $k = 15$, and, since the limit infimum is taken as the minimum, we take this value for $\rho$. Because $\gamma_k$ remains constant after $N = 3$, the constant $A_k$ naturally increases as the numerical estimate of $\rho$ improves. Therefore, $N_{\text{opt}}$ is effectively determined by the order at which the sequence $\rho_k$ flattens out, provided no deeper resonances are encountered. 

With $N_{\text{opt}}$, $C$, and the minimized constant $A$ identified and explicitly computed, we apply Theorem 8 and Corollary 9 to compute the confinement region. The formula in Corollary 9 works if we consider time scales on the order of $T > C$. For our orbit, $C \gtrsim \mathcal{O}(10^{50})$, which exceeds the known lifetime of the universe. Accordingly, and as we are interested in timespans on the order of mission lifetimes (e.g., 10-50 years), we apply the threshold condition of Proposition 6, which effectively replaces the $\frac{1}{\tau}\ln(T/C)$ with 1. Hence, we use the condition:
\begin{equation}
    a(T) \leq \begin{cases} 
        \frac{1}{Ae^\tau (1 + \frac{1}{\tau}\ln(\frac{T}{C}))^\tau}, & T \geq C \\
        \frac{1}{A(2e)^\tau}, & \text{otherwise}
    \end{cases}
\end{equation}

\begin{figure}[ht]
    \centering
    \begin{subfigure}{.5\textwidth}
        \centering
        \includegraphics[width=\linewidth,page=2]{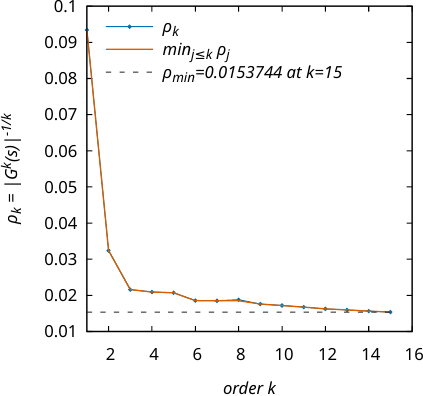}
        \caption{Radius of convergence of the un-normalized Taylor series $G(s)$ depending on the truncation order $k$, approximated using the Cauchy-Hadamard root test.}
        \label{fig:RoC_G}
    \end{subfigure}%
    \hfill
    \begin{subfigure}{.46\textwidth}
        \centering
        \includegraphics[width=\linewidth,page=3]{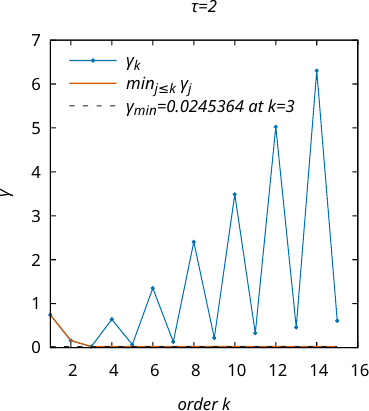}
        \caption{Diophantine constant $\gamma$ computed numerically as in Equation \ref{eqn:SmallDivisor} depending on the normalization order $k$ as the rolling minimum of the $\gamma_k$.}
        \label{fig:SmallDivisor}
    \end{subfigure}
    \caption{Computational results of radius of convergence and Diophantine constant depending on the order $k$.}
    \label{fig:ComputationSetup}
\end{figure}

Figure \ref{fig:StabilityRadius} shows the intermediate computation of the majorant coefficient $A$, as well as the effective stability radius, $a$ as a function of the order $k$. Figure \ref{fig:StabilityRegion} visualizes the confinement threshold in normal form coordinates, as well as physical spatial coordinates taken on the Poincaré section. Because operational timescales are vastly shorter than the Nekhoroshev constant ($T \ll C$), the exponential diffusion mechanism is entirely negligible. Consequently, the effective stability region for realistic mission lifetimes is not constrained by time-dependent drift, but is instead entirely bounded by the maximum analytical valid domain of the normal form itself. Thus, we find that for practical mission durations, the theoretical stability bounds intrinsically exceed operational needs, rendering this constant maximum radius the strict dynamical boundary of effective stability within the idealized CR3BP.

\begin{figure}[ht]
    \centering
    \begin{subfigure}{.43\textwidth}
        \centering
        \includegraphics[width=\linewidth,page=5]{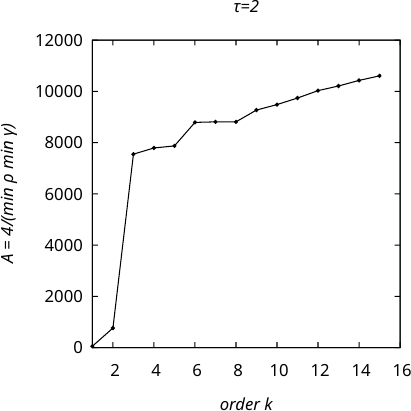}
        % \caption{Radius of convergence of the un-normalized Taylor series $G(s)$ depending on the truncation order $k$, approximated using the Cauchy-Hadamard root test.}
        \label{fig:A}
    \end{subfigure}%
    \hfill
    \begin{subfigure}{.5\textwidth}
        \centering
        \includegraphics[width=\linewidth,page=7]{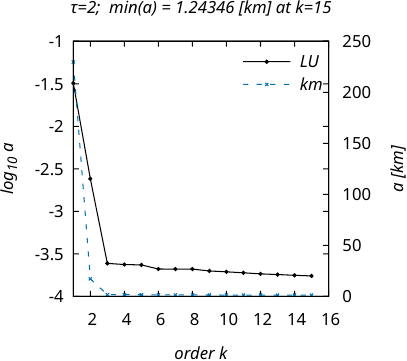}
        % \caption{Diophantine constant $\gamma$ computed numerically as in Equation \ref{eqn:SmallDivisor} depending on the normalization order $k$ as the rolling minimum of the $\gamma_k$.}
        \label{fig:a}
    \end{subfigure}
    \caption{Computational results of majorant coefficient, $A$, and effective stability radius, $a$, as a function of the order of normalization, $k$.}
    \label{fig:StabilityRadius}
\end{figure}

\begin{figure}[ht]
    \centering
\setlength{\tabcolsep}{0.5em}
\begin{tabular}{@{}ccc@{}}
  \includegraphics[width=.32\textwidth,valign=t,page=1]{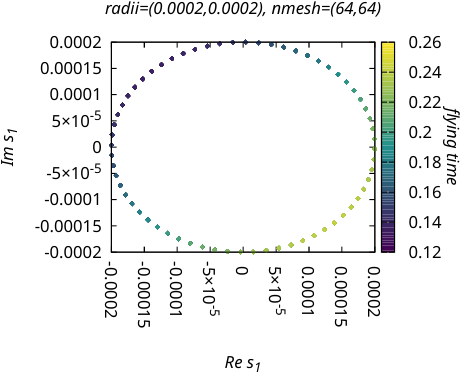} &
  \includegraphics[width=.32\textwidth,valign=t,page=2]{eval15-crop.pdf} &
  \includegraphics[width=.25\textwidth,valign=t,page=7]{eval15-crop.pdf} \\
  \\
  \includegraphics[width=.32\textwidth,valign=t,page=9]{eval15-crop.pdf} &
  \includegraphics[width=.32\textwidth,valign=t,page=10]{eval15-crop.pdf} &
  \includegraphics[width=.25\textwidth,valign=t,page=8]{eval15-crop.pdf}
\end{tabular}
    \caption{An approximated solution around the chosen $L_2$ Southern NRHO in the Earth-Moon system in Table~\ref{tab:OrbitData} on the Poincaré section $\Sigma$. First row in Normal Form coordinates, and second row in E-M coordinates.}
    \label{fig:StabilityRegion}
\end{figure}

\begin{table}[ht]
    \centering
    {\small 
    \begin{tabular}{c|c|cccc|c}
    $k$ & $\gamma _k = \min |\lambda _i - \lambda ^j| k ^\tau $ & $j _1$ & $j _2$ & $j _3$ & $j _4$ & $i$ \\ \hline
  1 & \texttt{ 7.457469939031671e-01} & \texttt{ 1} & \texttt{ 0} & \texttt{ 0} & \texttt{ 0} & \texttt{3} \\
 2 & \texttt{ 1.613326139608544e-01} & \texttt{ 1} & \texttt{ 0} & \texttt{ 1} & \texttt{ 0} & \texttt{2} \\
 3 & \texttt{ 2.453638511751952e-02} & \texttt{ 0} & \texttt{ 0} & \texttt{ 3} & \texttt{ 0} & \texttt{4} \\
 4 & \texttt{ 6.453304558434056e-01} & \texttt{ 3} & \texttt{ 0} & \texttt{ 1} & \texttt{ 0} & \texttt{1} \\
 5 & \texttt{ 6.815662532644190e-02} & \texttt{ 1} & \texttt{ 0} & \texttt{ 4} & \texttt{ 0} & \texttt{1} \\
 6 & \texttt{ 1.353866595655012e+00} & \texttt{ 1} & \texttt{ 0} & \texttt{ 5} & \texttt{ 0} & \texttt{2} \\
 7 & \texttt{ 1.335869856398272e-01} & \texttt{ 1} & \texttt{ 2} & \texttt{ 4} & \texttt{ 0} & \texttt{2} \\
 8 & \texttt{ 2.406873947831132e+00} & \texttt{ 2} & \texttt{ 1} & \texttt{ 5} & \texttt{ 0} & \texttt{2} \\
 9 & \texttt{ 2.208274660576034e-01} & \texttt{ 0} & \texttt{ 0} & \texttt{ 7} & \texttt{ 2} & \texttt{3} \\
10 & \texttt{ 3.488158750002063e+00} & \texttt{ 1} & \texttt{ 0} & \texttt{ 9} & \texttt{ 0} & \texttt{2} \\
11 & \texttt{ 3.298780665798766e-01} & \texttt{ 0} & \texttt{ 0} & \texttt{ 7} & \texttt{ 4} & \texttt{4} \\
12 & \texttt{ 5.022948600002852e+00} & \texttt{ 2} & \texttt{ 0} & \texttt{10} & \texttt{ 0} & \texttt{3} \\
13 & \texttt{ 4.607387872066040e-01} & \texttt{ 1} & \texttt{ 2} & \texttt{ 7} & \texttt{ 3} & \texttt{2} \\
14 & \texttt{ 6.302518131121055e+00} & \texttt{ 1} & \texttt{ 0} & \texttt{13} & \texttt{ 0} & \texttt{2} \\
15 & \texttt{ 6.134096279377864e-01} & \texttt{ 2} & \texttt{ 2} & \texttt{ 7} & \texttt{ 4} & \texttt{4}
    \end{tabular}}
    \caption{Explicit values of $\gamma_k$ (with $\tau =2$), multi-index $j$, and index $i$ for which the first non-resonant minimum is attained. Notice that because $|\lambda _i|=1$, the comparison for the minimum can be done in the angles of the eigenvalues in Table~\ref{tab:OrbitData}.}
    \label{tab:gammak}
\end{table}

In our numerical evaluation, the effective stability bounds and the optimal truncation order $N_{\text{opt}}$ are fundamentally dictated by the presence of a prominent low-order resonance. Specifically, the selected NRHO possesses a center mode with a Floquet exponent of approximately $\alpha_2 = \frac{\pi}{2} + \mathcal{O}(10^{-5})$. Consequently, the map's eigenvalues nearly satisfy the strong 1:4 normal resonance condition $\lambda_4 \approx \lambda_3^3$ (as $e^{-\frac{i\pi}{2}+\cdots} \approx e^{\frac{i3\pi}{2}+\cdots}$), as seen in Table \ref{tab:gammak}. This commensurability generates a severe small divisor early in the normalization process, which dominates the cumulative Diophantine constant $\gamma_k$ from order 3 onward. Moreover, this early normal resonance directly affects the optimized Nekhoroshev remainder and governs the finite-time stability bounds in the CR3BP. 

However, while this strong normal resonance condition dictates the effective stability limits in the autonomous system, it also reveals a critical dynamical insight for transitioning to higher-fidelity models. When attempting to transition this specific region of the NRHO band into the Elliptic Restricted 3-Body Problem (ER3BP), as in \cite{park2024characterizing}, the periodic orbit expands into a 2-dimensional invariant torus (for sufficiently low $e$ \cite{jorba_normal_1997}). The time-dependent external forcing associated with the Moon's eccentricity introduces a new base frequency to the system, acting as a second internal frequency governing quasi-periodic motion on the persistent torus. Because the unperturbed normal frequency is a rational multiple of the $2\pi$ period of the ER3BP, it couples with the newly introduced internal frequency. As demonstrated by Broer et al. \cite{BroerHJVW03}, the onset of such normal-internal resonances precipitates complex bifurcations and the rapid breakdown of quasi-periodic structures. Consequently, this specific normal-internal resonance mechanism warrants a more rigorous, dedicated investigation to better understand the boundaries of persisting NRHOs as they transition from the CR3BP into the ER3BP and other quasi-periodically forced intermediate cislunar models.

%%%%%%%%%%%%%%%%%%%
% VI. Conclusions %
%%%%%%%%%%%%%%%%%%%
\section{Conclusions}
In this work, we have rigorously evaluated the effective stability of the elliptic band of Near-Rectilinear Halo Orbits (NRHOs) within the Earth-Moon Circular Restricted 3-Body Problem (CR3BP). By leveraging jet transport, we computed high-order Taylor expansions of the Poincar\'e map centered on a selected orbit within this family. We then constructed discrete, finite-order normal forms using explicit polynomial conjugations to systematically isolate the integrable twist dynamics from the chaotic perturbations. By applying Cauchy estimates to the truncation remainder, we derived rigorous, Nekhoroshev-like bounds on the local dynamics. These analytical bounds provide strict, finite-time guarantees on the escape time of trajectories in the vicinity of elliptic NRHOs, successfully bridging the gap between localized linear stability and global nonlinear chaotic diffusion. 

While the methodology presented here offers a robust framework for quantifying effective stability, it is not without its limitations. First, our reliance on explicit polynomial compositions, while computationally efficient, relaxes the condition of exact symplecticity. Although the resulting artificial dissipation is strictly bounded by the truncation remainder and does not invalidate the Nekhoroshev estimates, it introduces conservative over-estimations in the escape times. Second, the Cauchy majorant techniques used to bound the infinite tail remainder inherently produce worst-case scenarios; the true physical stability of these orbits is likely much longer than the rigorously guaranteed lower bounds. Finally, the CR3BP remains a simplified dynamical model, and the stability boundaries derived herein do not explicitly account for higher-fidelity perturbations such as orbital eccentricity, solar gravity, or solar radiation pressure. 

These limitations naturally pave the way for several promising directions for future research:
\begin{enumerate}
    \item[1.] \textbf{Parameter Continuation:} A natural extension of this work is to track the evolution of these effective stability bounds continuously across the entire band of elliptic NRHOs (made possible through the use of jet transport, as in \cite{gimeno2025explicit},) or to analyze their sensitivity to variations in the mass parameter ($\mu$). Such a continuation study would precisely map out how the nonlinear stability radii expand or contract as the periodic orbits approach major bifurcations or transition into the linearly unstable regime. Additionally, this analysis could be applied to normally elliptic resonant periodic orbits persisting into periodically-perturbed models (e.g., the Elliptic R3BP), combining insights from Melnikov and Nekhoroshev theories. 
    \item[2.] \textbf{KAM Analysis and Perpetual Stability:} While this paper establishes finite-time effective stability, the topological existence of invariant tori bounding the center manifold remains an open question. Future work may seek to transition from Nekhoroshev-like stability bounds to rigorous Computer-Assisted Proofs (CAPs) of KAM tori, potentially utilizing the parameterization method to establish infinite-time stability bounds within specific sub-regions of the NRHOs \cite{cabre2003parameterization,haro2016parameterization,LlaveGJV05,figueras2017rigorous}.
    \item[3.] \textbf{Mission-Relevant Implications:} Most critically, the theoretical bounds derived here have practical implications for cislunar mission design, especially related to Lunar Gateway. By translating the complex radius of our stability estimates into physical position and velocity deviations, we establish rigorous bounds on the local dynamics in the CR3BP, providing a foundational baseline for designing operational corridors in higher-fidelity ephemeris models. Trajectories initialized within these bounds are mathematically proven to remain bounded over mission-relevant timescales (e.g., 10 to 50 years) without exhibiting catastrophic chaotic drift. Future efforts will integrate these nonlinear stability volumes with station-keeping algorithms to quantify the minimum $\Delta V$ required to artificially maintain a spacecraft within the effective stability radius when subjected to full-ephemeris perturbations. 
\end{enumerate}

\section*{Statements and Declarations}
The authors declare that they have no conflict of interest.

\section*{Acknowledgments}
The authors thank Mr. Cade Armstrong for generating Figures 1-2 of this manuscript, as well as Dr. Marc Jorba-Cusc\'o for insightful discussion regarding the canonicity of the normalizing transformations used herein. 

\smallskip

The project has been supported by the Spanish grants
PID2021-125535NB-I00 (MICINN/AEI/FEDER, UE), 
the Catalan grant 2021
SGR 01072, and by the Air Force Office of Scientific
Research under award number FA8655-24-1-7059.  The project that led to these results also received
the support of a fellowship from ``la Caixa'' Foundation (ID
100010434), the fellowship code is LCF/BQ/PR23/11980047.

\bibliographystyle{alpha}
\addcontentsline{toc}{section}{References}
{\small \bibliography{ds,references}}

\appendix

% % Additional comments about incorporating KAM 
% \section{Comments - Incorporating KAM}
% We can also use Sections 2.1-4 and 3 to devise a KAM theorem. These steps are: 
% \begin{enumerate}
%     \item Sections 2.1-4, 3 
%     \item Verify twist condition (which Joan can already do in the code) 
%     \item Smallness condition on the remainder. KAM requires a stronger smallness condition than Nekhoroshev, typically related to the Diophantine constant:
%     \begin{equation}
%         \|R^{(N)}\|_\rho \leq \varepsilon \ll \gamma^a \rho^b,
%     \end{equation}
%     with $\gamma$ the Diophantine constant and $a,b$ constants depending on the dimension and regularity (I think). We don't need to optimize the order of the normal form for KAM, but we need to optimize it for Nekhoroshev. In KAM, the order must be chosen large enough so the remainder is small. 
%     \item Apply a discrete-time KAM theorem, proving convergence of an infinite sequence of coordinate changes. 
%     \item Compute measure estimates for the surviving tori and bounds on the gaps near resonances. 
% \end{enumerate}
% This is likely too much work for the same paper, as 3-5 would be totally new steps for us, on top of all the Nekhoroshev parts. I think we can focus on the Nekhoroshev stability for this paper, and maybe at CELMEC we can discuss with some the KAM people how we can write a KAM paper with this approach. 

% \bibliography{references}% common bib file
% \input main.bbl

\end{document}